\documentclass[12pt, a4paper]{article}
\usepackage{amsfonts,amssymb,amsmath,amscd,latexsym,makeidx,theorem,graphics}
\usepackage{hyperref}
\usepackage{color}

\usepackage[pdftex]{graphicx}

\title{Gluing metrics with prescribed $Q$-curvature and different asymptotic behaviour in high dimension}
\author{Ali Hyder\thanks{The authors are supported by the Swiss National Science Foundation projects n. PP00P2-144669, PP00P2-170588/1 and P2BSP2-172064.} \\ {\small UBC Vancouver}\\ {\small \texttt{ali.hyder@math.ubc.ca} }\and Luca Martinazzi${}^*$ \\ {\small Universit\`a di Padova}\\ {\small \texttt{luca.martinazzi@math.unipd.it} }}

\newtheorem{thm}{Theorem}[section]
\newtheorem{thmx}{Theorem}[section]

\newtheorem{lem}[thm]{Lemma}
\newtheorem{prop}[thm]{Proposition}
\newtheorem{cor}[thm]{Corollary}
\newtheorem{defn}{Definition}[section]
\newtheorem{rem}{Remark}

\newenvironment{proof}{\noindent\textbf{Proof}}{\hfill$\square$\medskip}

\newcommand{\D}{\Delta}
\newcommand{\vp}{\varphi}
\newcommand{\R}{\mathbb{R}}
\newcommand{\ve}{\varepsilon}
\newcommand{\bra}[1]{\left(#1\right)}
\newcommand{\C}{\mathcal{C}}

\DeclareMathOperator{\loc}{loc}

\begin{document}

\maketitle

\begin{abstract} We show a new example of blow-up behaviour for the prescribed $Q$-curvature equation in even dimension $6$ and higher, namely given a sequence $(V_k)\subset C^0(\R^{2n})$ suitably converging we construct {for $n\ge 3$} a sequence $(u_k)$ of radially symmetric solutions to the equation
$${(-\Delta)^n u_k=V_k e^{2n u_k} \quad \text{in }\R^{2n},}$$
with $u_k$ blowing up at the origin \emph{and} on a sphere. We also prove sharp blow-up estimates. This is in sharp contrast with the $4$-dimensional case studied by F. Robert (J. Diff. Eq. 2006).

\medskip

\end{abstract}

 \section{Introduction to the problem}

Given a domain $\Omega\subset\R^{2n}$, we will consider sequence $(u_k)$ of solutions to the prescribed $Q$-curvature equation
\begin{align}\label{eq-1}
   (-\D)^nu_k= V_ke^{2nu_k}\quad \text{in }\Omega,
\end{align}
under the uniform (volume) bound
\begin{align}\label{eq-2}
\int_{\Omega}e^{2nu_k}dx\leq C,\quad k=1,2,3,\dots
\end{align}
and suitable bounds on $V_k\in L^\infty(\Omega)$.

Contrary to the two dimensional situation studied by Br\'ezis-Merle \cite{BM}, or the case of a compact manifold of dimension {$2n$} without boundary (see e.g. \cite{DR,Mal,Mar1,Ndi}), where blow up occurs only on a finite set $S_1$, in an open Euclidean  domain $\Omega$ of dimension $4$ or higher it is possible to have blow up on larger sets. More precisely, for a finite set $S_1\subset\Omega\subset\R^{2n}$ let us introduce
\begin{equation}\label{defK}
\mathcal{K}(\Omega, S_1):=\{\varphi\in C^\infty(\Omega \setminus S_1):\varphi\le 0,\,\varphi\not\equiv 0,\, \Delta^n \varphi\equiv 0\},
\end{equation}
and for a function $\varphi \in \mathcal{K}(\Omega,S_1)$ set
\begin{equation}\label{defS0}
S_\varphi:=\{x\in\Omega\setminus S_1: \varphi(x)=0\}.
\end{equation}

\begin{thmx}[Adimurthi-Robert-Struwe \cite{ARS}, Martinazzi \cite{Mar1}]\label{ARSM} 
Let $\Omega$ be a domain in $\R^{2n}$, $n>1$ and let $(u_k)$ be a sequence of solutions to \eqref{eq-1}-\eqref{eq-2}, where $V_k\rightarrow V_0>0$ locally uniformly in $\Omega$ for some $V_0\in C^0(\Omega)$,
 and define the set  (possibly empty) 
$$S_1:=\left\{x\in\Omega:\lim_{r\to0^+}\liminf_{k\to\infty}\int_{B_r(x)}V_ke^{2nu_k}dx\geq \frac{\Lambda_1}{2}\right\},\quad \Lambda_1:=(2n-1)!|S^{2n}|.$$ 
Then up to extracting a subsequence one of the following is true.
\begin{itemize}
 \item [i)] For every $0\leq \alpha<1$, $(u_k)$ converges in $C^{2n-1,\alpha}_{\loc}(\Omega)$.
 \item [ii)] There exists $\vp\in \mathcal{K}(\Omega, S_1)$ 
and a sequence of numbers $\beta_k\to\infty$, such that, setting $S:=S_1\cup S_\varphi$ 
we have 
\begin{equation}\label{convbetak}
\frac{u_k}{\beta_k}\to\vp\quad \text{in } C_{\loc}^{2n-1,\alpha}(\Omega\setminus S),
 \quad 0\leq \alpha<1.
\end{equation}
In particular $u_k\to-\infty$ locally uniformly in $\Omega\setminus S$.
 \end{itemize}
 \end{thmx}

\begin{defn} Given $u_k$, $S_1$ and $S_\varphi$ as in Theorem \ref{ARSM}, we shall call $S_1$ the \emph{concentration} blow-up set and $S_\varphi$ the \emph{polyharmonic} blow-up set. Similarly $x\in S_1$ is called a concentration blow-up point and $x\in S_\varphi$ a polyharmonic blow-up point.
\end{defn}

Recently the authors together with S. Iula proved a partial converse to Theorem \ref{ARSM}, which we state in a simplified form.

\begin{thmx}[Hyder-Iula-Martinazzi \cite{HIM}]\label{HIM}
Consider $0<\Lambda<\frac{\Lambda_1}{2}$, $\Omega\subset \R^{2n}$ open, $n\ge 2$, $\varphi\in \mathcal{K}(\Omega,\emptyset)$,
and a sequence $(V_k)\subset L^\infty(\Omega)$ with $0<a\leq V_k\leq b<\infty$.
Then there exists a sequence $(u_k)$ of solutions to  \eqref{eq-1} with 
\begin{align*}
\int_{\Omega}V_k e^{2nu_k}dx=\Lambda,
\end{align*}
such that $S_1=\emptyset$, \eqref{convbetak} holds  with $S=S_\varphi$ and $u_k \to +\infty$ locally uniformly on $S_\varphi$.
\end{thmx}

The main question is whether one can extend Theorem \ref{HIM} to include the case in which concentration and polyharmonic blow-up sets coexist,  i.e. $S_1\ne \emptyset$ and $S_\varphi\ne\emptyset$, which is a possibility left open in Theorem \ref{ARSM}. In this paper we shall address the radially symmetric case of dimension $2n\geq6$. By $C^\ell_{rad}(B_R)$ or $L^p_{rad}(B_R)$ we will denote the subspace of radially symmetric functions in $C^\ell(B_R)$ and $L^p(B_R)$.

\subsection{The blow-up analysis}

Let us first observe that in dimension $4$ the radial case has been completely described by F. Robert \cite{Rob}. 
 
\begin{thmx}[Robert \cite{Rob}]\label{thm-Rob} On $\Omega= B_R\subset\R^4$, $R>0$, let $(V_k)\subset C^0_{rad}(B_R)$ be a sequence converging to $V_0=6$ locally uniformly, and let $(u_k)$ 
 be a sequence of radial solutions to \eqref{eq-1}-\eqref{eq-2} with $n=2$. Then, up to extracting a subsequence, either $u_k$ converges in $C^3_{\loc}(B_R)$, or $u_k\to -\infty$ uniformly locally in $B_R\setminus \{0\}$, $V_k e^{nu_k}\rightharpoonup \kappa\delta_0$, for some $\kappa\in[0,16\pi^2]$ and we have one of the following blow-up behaviours:

\noindent i) $u_k(0)\le C$ for every $k$ and $\kappa=0$.\\
\noindent ii) $u_k(0)\to \infty$. Then if we set $r_k:=2e^{-u_k(0)}$ and $\eta_k(x):=u_k(r_kx)+\log(r_k)$, we have $3$ subcases.

\noindent ii.a) $\kappa=16\pi^2$ and $\eta_k\to\eta$ in $C^3_{\loc}(\R^4)$, where $\eta(x)=\log\left(\frac{2}{1+|x|^2}\right)$,

\noindent ii.b) $\kappa\in (0,16\pi^2)$ and $\eta_k\to \eta_\infty$ in $C^3_{\loc}(\R^4)$, 
where $\eta_\infty(x)= \frac{2\kappa}{\Lambda_1}\log|x|-a|x|^2+o(\log|x|)$ as $|x|\to\infty$ for some $a>0$.\\
\noindent ii.c) $\kappa=0$, $r_k^2\Delta u_k(0)\to -\infty$ and for $\tilde \eta_k:=\frac{\eta_k}{r_k^2\Delta u_k(0)}$ we have $\tilde\eta_k(x)\to \frac{|x|^2}{8}$ in $C^3_{\loc}(\R^4).$

Moreover in the cases ii.a) and ii.b) it holds
\begin{equation}\label{quant-Rob}
\lim_{L\to\infty}\lim_{k\to\infty}\int_{B_\delta\setminus B_{Lr_k}}V_ke^{4u_k}dx=0,\quad \text{for every }\delta<R.
\end{equation}
\end{thmx}
 

One of the crucial elements in the proof of Theorem \ref{thm-Rob} is the estimate
\begin{equation}\label{stimaRob}
|x|e^{u_k(x)}\le C(\delta),\quad \text{uniformly on }B_\delta\quad \text{for every }\delta<R.
\end{equation}
Unfortunately \eqref{stimaRob} does not hold in dimension $6$ (and higher), not even for small $\delta$, see e.g. Examples 2 and 3 in Section \ref{examples}, and in fact the blow-up behaviour is much richer, but we will not examine it in detail. Instead we will focus on the follwing particular issue.
By scaling one of the many entire solution to
\begin{equation}\label{entiresol}
{(-\Delta)^{2n} u=(2n-1)! e^{2nu}\quad \text{in }\R^{2n},\quad \Lambda:=(2n-1)!\int_{\R^{2n}}e^{2nu}dx<\infty}
\end{equation}
(see \cite{CC}) one finds a sequence $(u_k)$ of solutions to \eqref{eq-1}-\eqref{eq-2} blowing up at $0$. By Theorem \ref{HIM} one can also construct solutions blowing up on a sphere. The main question is whether one can ``glue'' the first kind of solutions to the second kind of solutions to obtain solutions with a concentration blow up at $0$ and a polyharmonic blow up on a sphere. 

In the following theorem, we show that in dimension 6 if this occurs (case $iv$), then the blow up at the origin is necessarily spherical, i.e. as in case $iia)$ of Theorem \ref{thm-Rob}.

\begin{thm}\label{thm-2}
Let $\Omega=B_R\subset\R^6$ for some $R>0$ (with $\Omega=\R^6$ if $R=\infty$). 
Let $(V_k)\subset L^\infty_{rad}(B_R)$ be positive radial functions with  $V_k\to V_0$ in $L^\infty_{\loc}(B_R)$ for some positive $V_0\in C^0(B_R)$. 
Let $(u_k)$ be radial solutions to  \eqref{eq-1}-\eqref{eq-2} with $n=3$. Assume that we are in case ii) of Theorem \ref{ARSM}. Then one of the following occurs.

\noindent i)  $S_1\cup S_\varphi=\emptyset$.

\noindent ii)  $S_1\cup S_\varphi=\{0\}$

\noindent iii) $S_1=\emptyset$ and $S_\vp=\{|x|=\rho\}$ for some $\rho\in (0,R)$. 

\noindent iv)  $S_1=\{0\}$ and $S_\vp=\{|x|=\rho\}$ for some $\rho\in (0,R)$. In this case, up to replacing $u_k(x)$ with $u_k(\rho x)+\log\rho$ and $R$ by $\rho R$, we can assume $\rho=1$ and, up to adding the constant $\tfrac16\log(\tfrac{V_0(0)}{120})$ to $u_k$, we can assume $V_0(0)=120$. Then $u_k(0)\to \infty$,
\begin{equation}\label{ukbetak}
\frac{u_k(x)}{\beta_k}\xrightarrow{k\to\infty} -\left(1-|x|^2\right)^2\quad\text{in }C^5_{\loc}(B_R\setminus\{0\})
\end{equation}
 for some $\beta_k\to\infty$ that satisfies
\begin{equation}\label{betakuk(0)}
\beta_k=o(e^{2u_k(0)}).
\end{equation}
Moreover for $r_k:=2e^{-u_k(0)}\to 0$ 
\begin{equation}\label{etaklog}
\eta_k(x):=u_k(r_kx)+\log(r_k)\xrightarrow{k\to\infty}\log\left(\frac{2}{1+|x|^2}\right)=:\eta(x) \quad \text{in }C^5_{\loc}(\R^6),
\end{equation}
Finally we have the following quantization result:
\begin{equation}\label{limcurv}
\lim_{k\to\infty}\int_{B_\delta}V_ke^{6u_k}dx=\Lambda_1\quad \text{for every } \delta\in(0,1).
\end{equation}
\end{thm}

One of the claims of Theorem \ref{thm-2} is that it is not possible to have $S_1=\{0\}$, $S_\varphi=\{|x|=\rho_1\}\cup \{|x|=\rho_2\}$ for some $0<\rho_1<\rho_2<R$, which is a priori not ruled out by Theorem \ref{ARSM}, since
$$\vp(r)=-\frac{1}{r^4}(\rho_1^2-r^2)^2(\rho_2^2-r^2)^2\in \mathcal{K}(B_R,\{0\}).$$
But the most important claim of Theorem \ref{thm-2} is that the profile of $u_k$ near the origin in case \emph{iv)} must be spherical, in the sense that it corresponds to the pull-back of the metric of $S^6$ onto $\R^6$ via the stereographic projection (compare to \cite{CC}).  For instance the behaviour seen in cases $iib)$ and $iic)$ of Theorem \ref{thm-Rob}, are possible in Theorem \ref{thm-2} in case $ii)$ but not in case $iv)$. The proof of \eqref{etaklog} will use the following classification result from \cite{Mar0}, see also \cite{H-clas,JMMX, Lin}:

\begin{thmx}[C-S. Lin \cite{Lin}, Martinazzi \cite{Mar0}]\label{thmclas}
Let $u$ be a solution to \eqref{entiresol}. Then either $u$ is spherical, i.e. it has the form $u(x)=\log \frac{2\lambda}{1+\lambda^2|x-x_0|^2}$ for some $\lambda>0$, $x_0\in\R^{2n}$, or $u=v+p$ where $p$ is  an upper bounded polynomial of degree at most $2n-2$  and $\Delta^j v(x)\to 0$ as $|x|\to\infty$ for $j=1,\dots,n-1$.
\end{thmx}
  

\subsection{The existence part}
 
According to Theorem \ref{thm-2}, the only case in which we can expect gluing of a \emph{concentration} blow up with a \emph{polyharmonic} blow up is case $iv$). That this can actually occur is the claim of the next theorem, which we state in dimension $6$, although it can be extended to any even dimension $2n\ge 6$, see Section \ref{dim2n}.

 \begin{thm}\label{thm-1} Let $(V_k)\subset C^0_{rad}(\R^6) $ be positive functions with $V_k\to V_\infty$  uniformly,    where $V_\infty\in C^1(\R^6)$, $V_\infty(0)=120$ (without loss of generality) and for some $a,b\ge 0$
 \begin{align}\label{condVinfty}
 \frac{d}{dr}\bra{\frac{V_\infty (r)}{e^{ar^2+br^4}}}\leq 0, \quad r\in [0,\infty).
 \end{align}
Then  
for every $\Lambda \ge\Lambda_1$ there exists a sequence $(u_k)\subset C^6_{rad}(\R^6)$ of solutions to \eqref{eq-1}-\eqref{eq-2} with $n=3$ and  $\Omega=\R^6$ such that $u_k$ is as in case iv) of Theorem \ref{thm-2} with $\rho=1$. 
Moreover $u_k(1)\to \infty$ as $k\to\infty$ and for every $\ve\in (0,1)$
\begin{align}
&\lim_{k\to\infty}\int_{B_{1+\ve}\setminus B_{1-\ve}}V_ke^{6u_k}dx=\Lambda-\Lambda_1,\label{curv1}\\
&\lim_{k\to\infty}\int_{\R^{6}\setminus (B_\ve\cup (B_{1+\ve}\setminus B_{1-\ve})}V_ke^{6u_k}dx=0.\label{curv2}
\end{align}
\end{thm}

This existence theorem will be based on various ingredients. First of all the following, slightly modified and simplified version of \cite[Theorem 1.1]{H-volume}:

\begin{thmx}[Hyder \cite{H-volume}]\label{fixed1}
Let $V\in C^0_{rad}(\R^{2n})\cap L^\infty(\R^{2n})$ be such that $V(0)>0$ and $V\geq 0$ in $\R^{2n}$, $n\ge 3$. Then for every $\Lambda>0$ there exists $u\in C^{2n}_{rad}(\R^{2n})$ 
solution to
$$(-\D)^nu=Ve^{2nu}\quad\text{in }\R^{2n},\quad \Lambda=\int_{\R^{2n}}Ve^{{2n}u}dx.$$
Moreover, for every $\lambda\in(0,1/8n]$ one can express $u$ in the form $u(x)=v(x)+c-|x|^4$ where $c\in\R$ and $v$ satisfies
$$v(x)=\frac{1}{\gamma_{2n}}\int_{\R^{2n}}\log\left(\frac{1}{|x-y|}\right)V(y)e^{-2n|y|^4}e^{2n(v(y)+c)}dy+\lambda\Delta v(0)(|x|^4-2|x|^2).$$
\end{thmx}

The above result, which also holds in odd dimension $5$ and higher, completely solves problems left open in \cite{JMMX,LM-vol,HD}, proving that in dimension $5$ and higher one can find conformally Euclidean metrics with constant $Q$-curvature and total $Q$-curvature arbitrarily large, in contrast to the $4$ dimensional case, where the total $Q$-curvature can be at most $\Lambda_1=6|S^4|$ as shown by \cite{Lin}. Theorem \ref{thm-1} is naturally related to Theorem \ref{fixed1} because in case $iv)$ of Theorem \ref{thm-2}, an amount $\Lambda_1$ of $Q$-curvature concentrates at the origin, and we expect to have additional curvature concentrating at $S_\varphi$.

The strategy in the proof of Theorem \ref{thm-2} will be to take an arbitrary sequence $\lambda_k\to 0^+$ in Theorem \ref{fixed1}, and find corresponding solutions $u_k(x)=v_k(x)+c_k-|x|^4$. Then with the help of a Pohozaev-type identity from \cite{Xu} and of a quantization result from \cite{Mar2} (see also \cite{Rob2}) one proves that $\lambda_k\Delta u_k(0)\to -\infty$ and this will finally lead to Theorem \ref{thm-1}.

\subsection{A sharper blow-up analysis in the hybrid case}

In order to obtain general existence results (see Section \ref{OQ} for open questions), possibly using the Lyapunov-Schmidt reduction, it might be useful to have precise information on a model case to use as ``ansatz''. In this spirit, pushing further the blow-up analysis of Theorem \ref{thm-2}, case $iv)$, we obtain sharp global estimates which relate the behaviour near the origin and the behaviour away from it.
Moreover, using a linearization procedure partly inspired from \cite{MM}, we are able to give a better asymptotic expansion of $u_k$ near the origin.

We will assume
\begin{equation}\label{assVk}
V_k(r)=120+O(r^2)\quad \text{as }r\to 0.
\end{equation}
As mentioned in Theorem \ref{thm-2}, the choice of the particular constant $120$ is not restrictive.

\begin{thm}\label{thm-exp} Assume that we are in case $iv)$ of Theorem \ref{thm-2} with $V_k$ additionally satisfying \eqref{assVk} and (up to a scaling) with $\rho=1$. Then for every $\delta\in (0,R)$ we have 
\begin{equation}\label{ukglobal}
u_k(r)=\bar\eta_k(r)+u_k(0)(\vp(r)+1+o(1))\quad \text{on }[0,\delta]
\end{equation}
with $o(1)\xrightarrow{k\to\infty} 0$, where $\bar\eta_k(r):=\eta(\tfrac{r}{r_k})-\log(r_k)=\log\left(\frac{2/r_k}{1+r^2/r_k^2}\right)$, and as a consequence
\begin{equation}\label{betakuk}
\beta_k=u_k(0)(1+o(1)).
\end{equation}
Moreover, setting 
$\ve_k := u_k(0)e^{-2u_k(0)}$ {and fixing}
$s_k>0$ such that $s_k=\ve_k^{-\frac14}o(1)$ we have an expansion
\begin{equation}\label{etakexp}
\eta_k(x)=\eta(x)+\ve_k\psi_0(x)+\ve_k(1+|x|^4) o(1),
\end{equation}
with $o(1)\xrightarrow{k\to\infty} 0$ uniformly for $x\in B_{s_k}$, where
\begin{equation}\label{stimapsi}
\psi_0(x)=8|x|^2-48\log|x|+o(\log|x|),\quad \text{as }|x|\to\infty.
\end{equation}
\end{thm}

A consequence of Theorem \ref{thm-exp} is a new phenomenon strongly related to the gluing of a concentration blow-up with a polyharmonic blow-up. While it is easy to construct a concentration blow up as $u_k(x)=\eta(x/r_k )+\log(1/r_k)$ with $\eta$ as in \eqref{etaklog},  
in this case using \eqref{defint} we have 
$$\int_{B_\delta} 120e^{6u_k}dx=\Lambda_1\bra{1-\frac{640}{\delta^6 e^{6u_k(0)}} +\frac{o(1)}{e^{6u_k(0)}} }.$$
i.e. in small neighborhoods of the origin the curvature concentrates to $\Lambda_1$ from below.
In the case of gluing with a polyharmonic blow up we obtain the opposite result. In this sense we see that the asymptotic behavior of the curvature concentrating at the origin is \emph{nonlocal}: it also depends on the behavior at larger scales.
 
\begin{thm}\label{thm-curv} Under the assumptions of Theorem \ref{thm-2}, case $iv)$ with $\rho=1$, and additionally assuming \eqref{assVk}, the limiting value $\Lambda_1$ in \eqref{limcurv} is reached from above. More precisely 
\begin{equation}\label{curvsharp}  
\int_{B_\delta}V_ke^{6u_k}dx=\Lambda_1+(24\Lambda_1+o(1))\frac{u_k(0)}{e^{2u_k(0)}}\quad   \text{for } 0<\delta < \delta^*:= \sqrt{1-\frac{1}{\sqrt 3}}. 
\end{equation}
\end{thm}

We can compare \eqref{curvsharp} with an analog energy expansion for the Moser-Trudinger equation given by Mancini and the second author \cite{ManMar} for the equation
\begin{equation}\label{eqMT}
-\Delta u=\lambda ue^{u^2} \quad \text{in }B_1\subset\R^2,\quad u=0\quad \text{on }\partial B_1,
\end{equation}
building upon \cite{MM}. A sequence $(u_k)$ of positive (hence radial, by the moving-plane technique) solutions to \eqref{eqMT} for some $\lambda_k>0$, with $u_k(0)\to\infty$ satisfies
\begin{equation}\label{MTexp}
4\pi +\frac{4\pi+o(1)}{u_k^4(0)}\le \int_{B_\delta}|\nabla u_k|^2dx\le 4\pi +\frac{6\pi+o(1)}{u_k^4(0)}.
\end{equation}
In spite of the similarities in the arguments, \eqref{MTexp} essentially depends on the Taylor expansion of the nonlinearity $ue^{u^2}$, which enjoys only an approximate scale invariance, contrary to the nonlinearity $e^{6u}$.

\paragraph{Notation}
We will often use the following constants:
\begin{equation}\label{constants}
\gamma_6 =2^6\pi^3, \quad \omega_6:=|S^6|=\frac{16}{15}\pi^3,\quad \omega_5:=|S^5|=\pi^3, \quad \Lambda_1 =2^7\pi^3 \quad \text{for }n=3,
\end{equation}
where $\Delta^3 \log |x|=\gamma_6\delta_0$ in $\R^6$. 

For sequences $(a_k)$ and $(b_k)$ with $b_k>0$ 
\begin{align*}
&a_k\approx b_k\quad\text{ if  }  \frac{1}{C}a_k\leq b_k\leq Ca_k,\\
&a_k=O(b_k )\quad\text{ if } | a_k|\leq C b_k, \\
&a_k=o(b_k) \quad\text{ if  }  \lim_{k\to\infty}\frac{a_k}{b_k}=0,
\end{align*}
 where  $C>0$ is independent of $k$. 
 
In the proofs we will often extract subsequences without explicitly mentioning it. Moreover, with a slight abuse of notation, we will use the notation $u(x)$ and $u(r)$ or $u(|x|)$ to denote the same radially symmetric function $u$.


\section{Proof of Theorem \ref{thm-2}}

\subsection{The possible blow-up sets $S$}

Consider Theorem \ref{ARSM}. Clearly either $S_1=\emptyset$ or $S_1=\{0\}$. We consider the two cases separately.

\medskip

\noindent\textbf{Case $S_1=\emptyset$.} It is not difficult to see that all radial functions in $\mathcal{K}(B_R,\emptyset)$ are of the form
$$\varphi(x)= a+b|x|^2+c|x|^4$$
and this easily leads to either $S_\varphi=\emptyset$ (case i) of the theorem) or $S_\varphi=\{0\}$ (case ii) of the theorem) or $S_\varphi=\{|x|=\rho\}$ for some $\rho\in (0,R)$ (case iii) of the theorem).

\medskip

\noindent\textbf{Case $S_1=\{0\}$.} If $S_\varphi=\emptyset$, then we are in case ii) of the theorem. If $S_\varphi\neq \emptyset$, since $S_\varphi\subset B_R\setminus \{0\}$ and $\varphi$ is radial, we can assume that $\{|x|=\rho\}\subset S_\varphi$ for some $\rho\in (0,R)$, and to simplify the notation we can assume, up to a scaling, $\rho=1$. In the next lemma we collect some important information about the sign and the zeroes of $u_k$, $\Delta u_k$ and $\Delta^2 u_k$ and their derivatives for $k$ large, compare to Figure \ref{graphs}.

 \begin{lem} \label{monotone}
Assume that
 \begin{align}
 S_1=\{0\}\quad \text{and} \quad \{|x|=1\}\subset S_\vp.  \label{eq-3} 
 \end{align}
Then for $k$ large there exist
$\theta_{1,k},\tilde \theta_{1,k}, \theta_{2,k},\tilde \theta_{2,k},\theta_{3,k},\theta_{4,k}\in (0,R)$
with
$$\theta_{2,k}<\theta_{1,k}<\tilde\theta_{1,k}\xrightarrow{k\to\infty} 1,\quad \theta_{2,k}<\theta_{3,k}<\tilde{\theta}_{2,k},\quad \theta_{4,k}<\theta_{3,k},$$
such that
$$u_k'(\theta_{1,k})=u_k'(\tilde \theta_{1,k})=\Delta u_k(\theta_{2,k})=\Delta u_k(\tilde \theta_{2,k}) =(\Delta u_k)'(\theta_{3,k})=\Delta^2 u_k(\theta_{4,k})=0$$
and the following holds.
\begin{align}
\D^2u_k>0& \quad\text{on }(0,\theta_{4,k}), \quad \D^2u_k<0\quad\text{on }(\theta_{4,k},R)\label{D2uk}\\
(\D u_k)'>0& \quad \text{on }(0,\theta_{3,k}), \quad (\D u_k)'<0\quad \text{on }(\theta_{3,k},R)\label{Duk'}\\
u_k'<0&\quad  \text{on }(0,\theta_{1,k})\cup (\tilde\theta_{1,k},R),\quad u_k'>0\quad \text{on }(\theta_{1,k},\tilde \theta_{1,k}).\label{uk'}
\end{align}
Moreover $u_k(0)\to\infty$, and $u_k\to-\infty$  uniformly on $(\theta_{1,k},1-\ve)\cup (1+\ve,R)$ for every $\ve>0$. Finally $S_\varphi=\{|x|=1\}$.
\end{lem}

\begin{proof} We will use that
\begin{equation}\label{Delta3neg}
\Delta^3 u_k<0, 
\end{equation}
which follows from \eqref{eq-1} since $V_k>0$, and repeatedly apply
\begin{align}\label{w'}
w'(r)=\frac{1}{\omega_5r^{5}}\int_{B_{r}}\D w(x)dx,\quad w\in W^{2,p}_{rad},
\end{align}
where we can take $p>6$, so that $w'\in C^0(0,R)$. We will also need \eqref{eq-2}, \eqref{convbetak} and
\begin{equation}\label{supukve}
\sup_{B_\ve} u_k\to\infty \quad \text{for every }\ve>0,
\end{equation}
which follows from $S_1=\{0\}$.

Since $\vp\leq 0$ in $(0,R)$ and $\vp(1)=0$, we can choose $\delta\in (0,1)$ such that
\begin{equation}\label{vpdelta}
\vp(1-\delta)<0, \quad\vp'(1-\delta)>0, \quad \vp(1+\delta)<0,\quad \vp'(1+\delta)<0. 
\end{equation}

From \eqref{Delta3neg} and \eqref{w'} we infer that
\begin{equation}\label{D2uk'}
(\Delta^2 u_k)'<0\quad \text{on }(0,R).
\end{equation}

\medskip

\noindent\textbf{Step 1} We claim that $\Delta^2 u_k(0)>0$. If this were not the case, using \eqref{w'} and \eqref{D2uk'} we would obtain $(\Delta u_k)'<0$ on $(0,R)$. We will see that this contradicts \eqref{vpdelta}. Indeed if $\Delta u_k(0)\le 0$, then $\Delta u_k< 0$ on $(0,R)$, hence by \eqref{w'} we have $u_k'<0$ on $(0,R)$, but then also $\varphi'\le 0$ on $(0,R)$, contradicting \eqref{vpdelta}. If $\Delta u_k(0)>0$, then $u_k'>0$ on $(0,t_k)$ for some $t_{k}\in (0,R]$. By \eqref{supukve} and \eqref{eq-2} we must have $t_k\to 0$. Using \eqref{w'} we then infer that $u_k'<0$ on $(t_k,R)$ hence $\varphi'<0$ on $(0,R)$, again contradicting \eqref{vpdelta}. Then Step 1 is proven.

\medskip

\noindent\textbf{Step 2} We claim that $(\Delta u_k)'$ changes sign only once from positive to negative, {there exists $\theta_{4,k}\in (0,R)$ such that $\Delta^2 u_k(\theta_{4,k})=0$, and $\Delta u_k$ has at most $2$ zeroes. Indeed, thanks to Step 1 and \eqref{w'} we know that $(\Delta u_k)'(r)>0$ for $r>0$ small. If $(\Delta u_k)'>0$ on $(0,R)$, again using \eqref{w'} with arguments similar to those of Step 1 we would obtain a contradiction. Using the monotonicity of $\Delta^2 u_k$ and \eqref{w'}, there must exist $\theta_{4,k}$ such that \eqref{D2uk} holds, and once $(\Delta u_k)'$ becomes negative, it remains so.}

\medskip

\noindent\textbf{Step 3} We claim that $\Delta u_k$ has exactly $2$ zeroes, $0<\theta_{2,k}<\tilde \theta_{2,k}<2$. Otherwise, considering Step 2, we would either have $\Delta u_k\le 0$, hence $u_k'\le 0$ on $(0,R)$, contradicting \eqref{vpdelta}, or $\Delta u_k(0)>0$, hence $u_k'\ge 0$ in a neighborhood of $0$ and then with \eqref{w'} and \eqref{eq-2} we see that $\theta_{1,k}\to 0$ and $u_k'<0$ on $(\theta_{1,k},R)$, contradicting \eqref{vpdelta}.

\medskip

\noindent\textbf{Step 4} {We claim that $u_k'$ has exactly $2$ zeroes $0<\theta_{1,k}<\tilde \theta_{1,k}<R$, so that \eqref{uk'} is satisfied. Indeed from Step 3 and \eqref{w'} it follows that $u_k'$ has at most $2$ zeroes, but using that $u_k'(1-\delta)>0$, $u_k'(1+\delta)<0$ (which follow from \eqref{convbetak} and \eqref{vpdelta}) and \eqref{supukve} we see that $u_k'$ must have at least $2$ zeroes.}

\medskip

\noindent\textbf{Step 5} {We claim that $u_k(0)\to\infty$ and $S_\varphi=\{|x|=1\}$. The first claim follows from Step 4 and \eqref{supukve}. The second one from Step 4 and \eqref{convbetak}, since if $\{|x|=1\}\cup\{|x|=\rho\}\subset S_\varphi$ for some $1\ne \rho\in (0,R)$, then $u_k'$ would have at least $4$ zeroes in $(0,R)$.}

\medskip


 
\noindent\textbf{Step 6} { To conclude it remains to observe that $\tilde\theta_{1,k}\to1$, which easily follow from Step 5, and that $u_k\to-\infty$ uniformly in $(\theta_{1,k}, 1-\ve)\cup (1+\ve,R)$ for every $\ve>0$, which follows from $\vp(r)<0$ on $(0,R)\setminus\{1\}$ and from \eqref{uk'}. }
\end{proof}

We have therefore proven that only the 4 given cases in Theorem \ref{thm-2} can occur. In the next subsections we shall focus on case $iv)$ and prove \eqref{ukbetak}, \eqref{betakuk(0)}, \eqref{etaklog} and \eqref{limcurv}.

\begin{figure}
\caption{The graphs of $u_k$, $\Delta u_k$ and $\Delta^2 u_k$ in the case $S_1=\{0\}$, $S_\varphi=\{|x|=1\}$.}\label{graphs}
 \includegraphics[width=7.5cm]{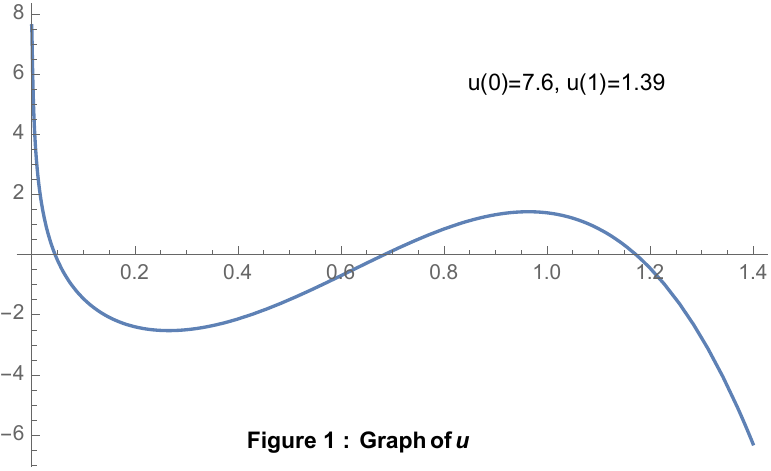}  \rule{1cm}{0cm}   \includegraphics[width=7.5cm]{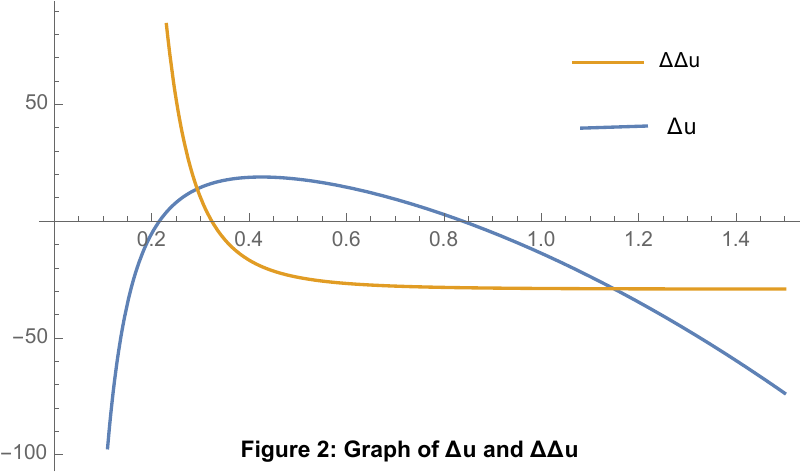} 
\end{figure}

\subsection{Proof of \eqref{ukbetak}, \eqref{betakuk(0)} and \eqref{etaklog}}

We shall now assume that we are in case $iv)$ of Theorem \ref{thm-2}, i.e. $S_1=\{0\}$ and $S_\varphi=\{|x|=1\}.$

\begin{lem}\label{stimeDelta} We have
$$\D^2 u_k(r)\leq \frac{C}{r^4} \quad \text{for }r\in (0,R)$$
and
$$0<-\D u_k(r)\leq \frac{C}{r^2} \quad \text{for }r\in (0,\theta_{2,k}),$$
where $\theta_{2,k}$ is as in Lemma \ref{monotone}. 
\end{lem}

\begin{proof} Using \eqref{w'} and the fundamental theorem of calculus we write for $|x|\leq \theta_{4,k}$
\[\begin{split}
\D^2 u_k(x)&= \int_{|x|}^{\theta_{4,k}}\frac{1}{\omega_5r^5}\int_{B_r}V_k(y)e^{6u_k(y)}dydr\\
&\leq C \int_{|x|}^{\theta_{4,k}}\frac{1}{r^5}dr\leq \frac{C}{|x|^4},
\end{split}\] 
and for $|x|>\theta_{4,k}$ the inequality is obvious. This yields for $|x|\leq \theta_{2,k}$
\begin{align*}
 -\D u_k(x)=\int_{|x|}^{\theta_{2,k}}\frac{1}{\omega_5r^5}\int_{B_r}\D^2u_k(y)dydr\leq \frac{C}{|x|^2}.
\end{align*}
\end{proof}

The following estimates can be seen as an extention of Lemma 3.5 in \cite{Rob}, whose method of proof goes back to \cite{RS}.
 
 \begin{lem}\label{derivative-uk}
 Let $0<\delta<1$ be fixed. Then there exists $C=C(\delta)>0$ ($C(\delta)$ can be made independent of  $\delta$ if $\delta$ lies in a compact subset of $[0,1)$)  such that  
  \begin{itemize}
  
   \item[i)] $re^{u_k(r)}\leq C$ for every $0\leq r\leq\delta$.  
  \item[ii)] $  |(\D u_k)'(r)-\frac r6\D^2u_k(\delta)|\leq \frac{C}{r^3},\quad 0< r\leq\delta.$
\item[iii)] $\left| u_k'(r)+\delta^2\frac{\D^2u_k(\delta)}{72}r-\frac{\D^2u_k(\delta)}{96}r^3-\frac{\D u_k(\delta)}{6}r\right| \leq \frac{C}{r} ,\quad 0< r\leq\delta.$ 

 \end{itemize}
\end{lem}
\begin{proof}
  Since $u_k\to-\infty$ uniformly in $(\theta_{1,k},\delta)$, we have $i)$ for $\theta_{1,k}\leq r\leq\delta$. For $0<r<\theta_{1,k}$ we use that $u_k$ is monotone decreasing. Indeed,
 $$r^6e^{6u_k(r)}\leq C\int_{B_{r}}e^{6u_k(y)}dy\leq C.$$
 To prove $ii)$ and $iii)$ we use Greens representation formula. We have 
 \begin{align} \label{25}
  \D u_k(x)&=-\int_{B_\delta}G(x,y)\Delta^2u_k(y)dy+\Delta u_k(\delta)\notag\\
  &=-\int_{B_\delta}V_ke^{6u_k(z)}\left(\int_{B_\delta}G(x,y)G(y,z)dy\right)dz-\D^2u_k(\delta)\int_{B_\delta}G(x,y)dy+\D u_k(\delta)\notag\\
  &=-\int_{B_\delta}V_ke^{6u_k(z)}\left(\int_{B_\delta}G(x,y)G(y,z)dy\right)dz-\D^2u_k(\delta)\frac{\delta^2-|x|^2}{12}+\D u_k(\delta),
 \end{align}
where $G$ is the Green function for $-\D$ on $B_\delta$ with Dirichlet boundary condition. Hence, together with $i)$
  \begin{align}
  \left|(\D u_k)'(r)-\frac r6\D^2u_k(\delta)\right|&\leq C\int_{B_\delta}e^{6u_k(z)}\left(\int_{B_\delta}|\nabla_xG(x,y)|G(y,z)dy\right)dz,\quad r:=|x|\notag\\
  &\leq C\int_{B_\delta}e^{6u_k(z)}\int_{B_\delta}\frac{dy}{|x-y|^5|y-z|^4}dz   \notag\\ 
  &\leq C\int_{B_\delta}e^{6u_k(z)}\frac{1}{|x-z|^3}dz .  \notag 
 \end{align}
 For $x\neq 0$ we split the domain $B_\delta$ into 
$$B_\delta=\cup_{i=1}^3A_i,\quad A_1:=B_\frac{|x|}{2},\quad A_2={(B_{2|x|}\cap B_\delta)}\setminus A_1,\quad A_3:=B_\delta\setminus(A_1\cup A_2).$$ Using that $$\frac{1}{|x-z|}\leq \frac{2}{|x|} \quad \text{for every }z\in A_1\cup A_3,$$ and together with $i)$, we bound 
 $$ \int_{B_\delta}e^{6u_k(z)}\frac{1}{|x-z|^3}dz\leq \frac{C}{|x|^3}\int_{A_1\cup A_3}e^{6u_k(z)}dz+C\int_{A_2}\frac{1}{|z|^6}\frac{1}{|x-z|^3}dz\leq\frac{C}{|x|^3}. $$
{This proves $ii)$.} From the identity \eqref{w'} and by \eqref{25} one gets
\begin{align*}
&\left| u_k'(r)+\delta^2\frac{\D^2u_k(\delta)}{72}r-\frac{\D^2u_k(\delta)}{96}r^3-\frac{\D u_k(\delta)}{6}r\right| \\
&\leq \frac{1}{\omega_5r^5}\int_{B_r}\int_{B_\delta}V_ke^{6u_k(z)}\left(\int_{B_\delta}G(x,y)G(y,z)dy\right)dzdx \\
&\leq \frac{C}{r^5}\int_{B_r}\int_{B_\delta}e^{6u_k(z)}\frac{1}{|x-z|^2}dzdx \\
&\leq \frac{C}{r^5}\int_{B_r}\frac{1}{|x|^2}dx \\
&\leq \frac{C}{r},
\end{align*}
{hence also $iii)$ is proven.}
\end{proof}

\paragraph{Proof of \eqref{ukbetak} (completed).} 

By assumptions $\frac{u_k}{\beta_k}\to \varphi$ in $C^5_{\loc}(B_R\setminus (\{0\}\cup\{|x|=1\}))$.
We know that
   $\varphi\in C_{rad}^\infty(B_R\setminus \{0\})$ solves the ODE $$ \Delta^3\varphi=0 \quad \text{in }(0,R),\quad  \vp(1)=0 \quad \vp'(1)=0,\quad \vp\not\equiv 0\quad \text{and }\vp\leq 0.$$ 
Therefore, $\vp$ is of the form 
 \begin{align}
  \varphi(r)=c_1+c_2r^2+c_3r^4+\frac{c_4}{r^4}+\frac{c_5}{r^2}+c_6\log r,\label{phi}
 \end{align}
 for some constants $c_0,\dots, c_6$.
This yields $$ \vp'(r)+\delta^2\frac{\D^2\vp(\delta)}{72}r-\frac{\D^2\vp(\delta)}{96}r^3-\frac{\D \vp(\delta)}{6}r=-\frac{4c_4}{r^5}-\frac{2c_5}{r^3}+\frac{c_6}{r}+\frac{2c_5r}{3\delta^4}-\frac{8c_5r}{9\delta^2}+\frac{c_6r^3}{6\delta^4}.$$ 
Dividing by $\beta_k$ in $iii)$ of Lemma \ref{derivative-uk} and using \eqref{convbetak}
we obtain $$-\frac{4c_4}{r^5}-\frac{2c_5}{r^3}+\frac{c_6}{r}+\frac{2c_5r}{3\delta^4}-\frac{8c_5r}{9\delta^2}+\frac{c_6r^3}{6\delta^4}\equiv0,$$ which implies that $c_4=c_5=c_6=0$.   Now the condition $\vp(1)=\vp'(1)=0$  gives  $c_1=c_3$ and $c_2=-2c_3$, that is, $\vp(r)=c_3(1-r^2)^2$. Since $\varphi\le 0$ we must have $c_3<0$, and up to replacing $\beta_k$ with $|c_3|\beta_k$ we obtain $c_3=-1$.

It remains to prove that the convergence in \eqref{ukbetak} holds in $B_R\setminus\{0\}$ (and not just in $B_R\setminus (\{0\}\cup\{|x|=1\})$). It follows easily from the monotonicity of $\Delta^2 u_k$ and from $\frac{\Delta^2 u_k}{\beta_k}\to \Delta^2\varphi$ uniformly locally in $B_R\setminus (\{0\}\cup\{|x|=1\})$ that
$$\frac{\Delta^2 u_k}{\beta_k}\to \Delta^2\varphi\quad \text{uniformly locally in }B_R\setminus \{0\}.$$ 
Then using   \eqref{int-repre} we obtain 
$$\frac{\Delta u_k}{\beta_k}\to \Delta \varphi,\quad \frac{u_k}{\beta_k}\to\varphi\quad \text{uniformly locally in }B_R\setminus \{0\},$$ 
and using \eqref{eq-1} and \eqref{w'} again, we also infer the $C^5$-convergence claimed in \eqref{ukbetak}.
\hfill$\square$

\medskip

In the following $\beta_k\to\infty$ is such that we have \eqref{ukbetak}. The following is a simple consequence of Lemma \ref{monotone} and \eqref{ukbetak}.

\begin{cor}\label{corthetaik} We have $\theta_{i,k}=o(1)$ for $i=1,\dots,4$, where $\theta_{i,k}$ is as in Lemma \ref{monotone}.
\end{cor}



\begin{lem}\label{lemetaklog}
Let $(u_k)$ be radial solutions to \eqref{eq-1}-\eqref{eq-2} in $\Omega= B_{R_0}$, $R_0>0$, with $V_k$ as in Theorem \ref{thm-2}, satisfying
\begin{equation}\label{ukmax12}
u_k(x)\leq u_k(0)\quad\text{ on }B_\varepsilon,
\end{equation}
for some $\varepsilon\in (0,R_0)$.
Assume further that $u_k(0)\to\infty$, $\Delta u_k(0)\le 0$ and that there are $\alpha_k\to\infty$, $\tau\in (0,R_0)$ and constants $C_1$ and $C_2$ (depending on $\tau$) such that
\begin{align}
\Delta u_k(\tau)&=\alpha_k(C_1+o(1)) \label{Dukdelta}\\
\Delta^2u_k(\tau) &=\alpha_k(C_2+o(1)) \label{D2ukdelta},
\end{align}
and
\begin{equation}\label{DD2ukdelta}
C_3:=C_1-\frac{C_2\tau^2}{12}>0.
\end{equation}
Then for $r_k:=2e^{-u_k(0)}$ we have $\alpha_k r_k^2=o(1)$ and \eqref{etaklog} holds.
\end{lem}

\begin{proof} From \eqref{ukmax12} we infer
\begin{equation}\label{boundetak}
\eta_k\le \eta_k(0)=\log 2 \quad \text{for }|x|\le \frac{\varepsilon}{r_k},
\end{equation}
hence
\begin{equation}\label{boundDetak}
|(-\Delta)^3\eta_k(x)|\le C  \quad \text{for }|x|\le \frac{\varepsilon}{r_k}.
\end{equation}
We now want to use \eqref{boundDetak} together with elliptic estimates applied to the function $\Delta \eta_k$ and then to $\eta_k$.
With $\tau \in (0,1)$ fixed such that \eqref{Dukdelta}-\eqref{DD2ukdelta} hold,  we obtain from \eqref{25}
 \begin{align*}
 &\D \eta_k(x)+r_k^2\left(\frac{\D^2u_k(\tau)\tau^2}{12}-\D u_k(\tau)\right)\\&=r_k^4\frac{\D^2u_k(\tau)|x|^2}{12}-r_k^2\int_{B_\tau}e^{6u_k(z)}V_k(z)\int_{B_\tau}G(r_kx,z)G(y,z)dydz. 
 \end{align*}
and integrating on $B_R$
 \begin{align}
 \int_{B_R} |\D \eta_k(x)-r_k^2\alpha_k(C_3+o(1))|dx&\leq Cr_k^4\alpha_kR^8+C\int_{B_R}\int_{B_\tau}e^{6u_k(z)}V_k(z)\frac{r_k^2}{|r_kx-z|^2}dzdx\notag\\
 &\leq Cr_k^4\alpha_kR^8+CR^4.\label{boundD1etak}
 \end{align}

We now claim that $\limsup_{k\to\infty} r_k^2\alpha_k<\infty$.
Indeed, assume by contradiction that for a subsequence $r_k^2\alpha_k\to\infty$. Set $\tilde\eta_k:=\frac{\eta_k}{r_k^2\alpha_k}$. Then by \eqref{boundDetak}, \eqref{boundD1etak} and elliptic estimates $\Delta \tilde \eta_k$ is uniformly bounded in $L^\infty_{\loc}(\R^6)$, and using \eqref{boundetak} and the Harnack inequality one has $\tilde\eta_k\to\tilde\eta$ in $C^5_{\loc}(\R^6)$ where $\tilde\eta$ satisfies
$$\D^3\tilde\eta=0\quad\text{in }\R^6,\quad \int_{B_R}|\D \tilde\eta(x)-C_3|dx=0.$$
This shows that $\D\tilde\eta(0)=C_3>0$, which contradicts $\D u_k(0)\leq 0$. This proves our claim.

Now, up to a subsequence we set $a:=\lim_{k\to\infty}r_k^2\alpha_k<\infty$. With the same elliptic estimates used for $\tilde\eta_k$ we get $\eta_k\to\eta_\infty$ in $C^5_{\loc}(\R^6)$ where $\eta_\infty$ satisfies 
$$(-\D)^3\eta_\infty=120e^{6\eta_\infty}\quad\text{in }\R^6,\quad \int_{\R^6}e^{6\eta_\infty}dx<\infty. $$ Moreover, 
\begin{equation}\label{eqDeltaeta}
\int_{B_R}|\D \eta_\infty(x)-C_3a|dx\leq CR^4\quad\text{for every }R>0.
\end{equation}
By Theorem \ref{thmclas} we can write $\eta_\infty=v+p$ with $\Delta v(x)\to 0$ as $|x|\to\infty$ and $p$ is a (radially symmetric) upper bounded polynomial of degree at most $4$. In particular $\lim_{|x|\to\infty} \Delta p(x)\le 0$. Since $a\geq 0$, from \eqref{eqDeltaeta} we infer that $a=0$, which is only possible if $\Delta p \equiv 0$, that is, $p$ is constant. By Theorem \ref{thmclas}, also observing that $\eta_\infty(0)=\log 2$, we conclude that   $\eta_\infty=\eta$, so that \eqref{etaklog} is proven. 
\end{proof}

\paragraph{Proof of \eqref{betakuk(0)} and \eqref{etaklog} (completed)}
It follows from Lemma \ref{monotone} and \eqref{ukbetak} that \eqref{ukmax12} holds.
From \eqref{ukbetak} we get \eqref{Dukdelta}-\eqref{DD2ukdelta} with
$$\alpha_k=\beta_k,\quad C_1=24-32\delta^2,\quad C_2=-384, \quad C_3=24,$$
for any $\delta\in (0,1)$, so that \eqref{etaklog} follows at once from Lemma \ref{lemetaklog}.
Moreover, the claim $\alpha_k r_k^2=o(1)$ of Lemma \ref{lemetaklog} is equivalent to \eqref{betakuk(0)}.
\hfill$\square$


\subsection{Proof of \eqref{limcurv}}

 \begin{lem}\label{neck1}
 Let $p\in (1,2)$, $\delta\in (0,1)$ be fixed. Let the assumptions of Lemma \ref{lemetaklog} be in force and additionally assume that there exists $0<\theta_k=o(1)$ such that
 \begin{equation}\label{-uk'}
0\leq -u_k'(r)\leq \frac{C}{r}\quad \text{on }(0,\theta_k)\text{ and} \quad u_k'(r)\geq 0\quad\text{on } (\theta_k, \delta).
\end{equation}
Then for each $k$ large there exists $t_k\in (0,\theta_k]$ such that  following hold:
 \begin{itemize}
 \item[i)] $r^pe^{u_k}$ is monotone decreasing on $(c_pr_k,t_k)$ for some constant $c_p>0$.
 \item[ii)]   $r_k=o(t_k)$ as $k\to\infty$, 
 \item[iii)] $u_k(t_k)\leq u_k(\theta_{k})+C.$
\end{itemize}
Finally, if $u_k(\delta)\to-\infty$, we conclude
\begin{equation}\label{limcurvbis}
\lim_{k\to\infty} \int_{B_\delta}V_k e^{6u_k}dx=\Lambda_1.
\end{equation}
\end{lem}

\begin{proof}
{For the proof of $i)$ and $ii)$ we shall follow \cite{Rob}.}
 We set  $c_p=\sqrt{1+\frac{p}{2-p}}$. For any $L>c_p$ and for $r\in (c_pr_k,Lr_k)$, using \eqref{etaklog}, which follows from Lemma \ref{lemetaklog}, we get
 \begin{align*}
  (r^pe^{u_k(r)})'(r)&=\left(pr^{p-1}+r^pu_k'(r)\right)e^{u_k(r)}\\
  &=r^{p-1}\left(p+\rho r_ku_k'(\rho r_k)\right)e^{u_k(r)},\quad \rho:=\frac{r}{r_k}\in (c_p,L)\\
  &=r^{p-1}\left(p+\rho \left(\frac{-2\rho}{1+\rho^2}+o(1)\right)\right)e^{u_k(r)}\\
  &<r^{p-1}\left(\frac{p-2}{1+\rho^2}+o(1)\right)e^{u_k(r)}\\
  & <0,
 \end{align*}
where $o(1)\to 0$ as $k\to\infty $ uniformly on $\rho\in (c_p,L)$. We set 
$$t_k:=\inf \{r\in (c_pr_k,\theta_{k}): (r^pe^{u_k(r)})'(r)=0\}.$$ It is easy to see that $t_k$ is well defined, $r_k=o(t_k)$ and   $r^pe^{u_k(r)}$ is monotone decreasing on $(c_pr_k,t_k)$.   

\medskip

Now we prove $iii)$ in few steps. 


\medskip\noindent\textbf{Step 1} $t_ke^{u_k(t_k)}\to 0$.

It follows from $i)$ that  $re^{u_k}$ is monotone decreasing on $(c_pr_k, t_k)$. Using that $r_k=o(t_k)$ and \eqref{etaklog}
we obtain for any $L>c_p$ and for $k$ large
\begin{align}
t_k e^{u_k(t_k)}\leq Lr_ke^{u_k(Lr_k)}=L\left(\frac{2}{1+L^2}+o(1)\right),\quad o(1)\xrightarrow{k\to\infty}0. \label{uktk}
\end{align}
Taking $k\to\infty$ and then taking $L\to\infty$ one has Step 1. 

\medskip\noindent\textbf{Step 2} There exists $C>0$ such that  $\theta_{k}\leq Ct_k.$

We assume by contradiction that $\frac{\theta_{k}}{t_k}\to \infty$. Then for any $\ve>0$ we have  $u_k'(rt_{k})<0$ for $r\in (\ve,\tfrac1\ve)$ and $k\ge k_0(\ve)$ large, thanks to \eqref{-uk'}.
Then, setting
$$\bar u_k(r)=u_k(rt_k)-u_k(t_k).$$ 
we get for $k\ge k_0(\ve)$
\begin{align}\label{8}
0< -\bar u_k'(r)\leq \frac{C}{r},\quad r\in (\ve, \tfrac1\ve). 
\end{align}
Hence, $\bar u_k\to \bar u_\infty$ in $C^0_{\loc}(0,\infty)$. By \eqref{uktk} we have
$$(-\Delta)^3 \bar u_k(r)=V_k(t_k r)t_k^6e^{6u_k(t_k)}e^{6\bar u_k}=o(1)e^{6\bar u_k}$$
with $o(1)\to 0$ locally uniformly for $r\in [0,\infty)$, thanks to Step 1. Then, by elliptic estimates $\bar u_k\to \bar u_\infty$
also in $C^5_{\loc}(0,\infty)$  
where $\Delta^3 u_\infty\equiv 0$ in $\R^6\setminus\{0\}$
$$|\bar u_\infty (r)|\leq \frac{C}{r},\text{ for }r\in(0,\infty), \quad \bar u_\infty(1)=0,\quad \bar u_\infty'(1)=-p.$$
Since  $\bar u_\infty$ is radial, it is of the form given in \eqref{phi},  and hence $\bar u_\infty (r)=-p\log r$. Using that $r_k=o(t_k)$ and \eqref{etaklog} one has
\begin{align*}
 64p+o(1)=-t_k^5(\D^2u_k)'(t_k)=\frac{1}{\omega_5}\int_{B_{t_k}}V_ke^{6u_k}dx\geq \frac{\Lambda_1}{\omega_5} +o(1)=128+o(1),
\end{align*}
a contradiction as $p<2$.

\medskip\noindent\textbf{Step 3}  $u_k(t_k)\leq u_k(\theta_{k})+C$. 

Since taking $\ve$ sufficiently small \eqref{8} holds for every $r\in (\ve, \tfrac{\theta_{k}}{t_k})$, we have 
\begin{align*}
u_k(t_k)&= u_k(\theta_{k}) -\int_{1}^\frac{\theta_{1,k}}{t_k}\bar u_k'(r)dr\\
&\leq u_k(\theta_{k})+C\log \frac{\theta_{k}}{t_k}\\
&\leq u_k(\theta_{k})+C,
\end{align*}
thanks to Step 2.

\medskip\noindent\textbf{Step 4} \eqref{limcurvbis} holds for $\delta$ such that $u_k(\delta)\to-\infty$.

Since $u_k(\theta_{k})<u_k(\delta)$ for $k$ large, we conclude
$$u_k(t_k)\le u_k(\theta_{k})+C\le u_k(\delta)+C\to -\infty \quad \text{as }k\to\infty.$$
Splitting the domain $B_\delta$ into  $$B_\delta=\cup_{i=1}^3A_i,\quad A_1:=B_{Lr_k},\quad A_2:=B_{t_k}\setminus B_{Lr_k}, \quad A_3:=B_\delta\setminus B_{t_k},$$ we write 
  \begin{align*}
  \int_{B_\delta}V_ke^{6u_k}dx=\sum_{i=1}^3I_i,\quad I_i:=\int_{A_i}V_ke^{6u_k}dx. 
  \end{align*}
Using the monotonicity of $u_k$ we infer
$$\max_{[t_k,\delta]}u_k=\max\{u_k(t_k),u_k(\delta)\}\xrightarrow{k\to\infty}-\infty,$$
which gives $I_3\to0$ as $k\to\infty$. For $L$ large we have $r^pe^{u_k(r)}$ monotone decreasing on $(Lr_k,t_k)$, and hence 
  \begin{align*}
  I_2\leq C\int_{A_2}|x|^{6p}e^{6u_k(x)}\frac{1}{|x|^{6p}}dx\leq C (Lr_k)^{6}e^{6u_k(Lr_k)}\xrightarrow{k\to\infty}C\left(\frac{L}{1+L^2}\right)^6,
\end{align*}
so that
$$\lim_{L\to\infty} \lim_{k\to\infty} I_2=0.$$
Finally   
$$\lim_{L\to\infty} \lim_{k\to\infty} I_1=\Lambda_1$$
by \eqref{etaklog}.
The proof of \eqref{limcurvbis} follows immediately.
\end{proof}

\paragraph{Proof of \eqref{limcurv} (completed).}
In the proof of \eqref{etaklog} we have already verified that the assumptions of Lemma \ref{lemetaklog} are in force. We claim that also \eqref{-uk'} holds with $\theta_k=\theta_{1,k}$, where $\theta_{1,k}$ is given by Lemma \ref{monotone}.

Indeed for $k$ sufficiently large, since $\D \vp(\tfrac12)>0$, \eqref{ukbetak} implies that $\D u_k(\tfrac12)>0$, hence by \eqref{Duk'} we have $\D u_k>0$ on $(\theta_{2,k},\frac 12]$. Together with Lemma \ref{stimeDelta} this gives
$$-\D u_k(r)\leq\frac{C}{r^2} \quad \text{on }(0,\tfrac12],$$
hence $-u_k'\le \frac{C}{r}$ on $(0,\tfrac12]$. Then \eqref{-uk'} follows.

Now observe that \eqref{ukbetak} implies $u_k(\delta)\to-\infty$ for $\delta\in (0,1)$. Then \eqref{limcurv} follow from Lemma \ref{neck1}
 \hfill$\square$


\section{Proof of Theorem \ref{thm-1}}


Let  $V_k$, $V_\infty$ and $P(r):=-ar^2-br^4$ be as in the statement of Theorem \ref{thm-1}. If we can find $\tilde u_k$ satisfying the requests of the theorem with $V_k$ replaced by $\tilde V_k:=V_ke^P$, then $u_k:=\tilde u_k+\frac{P}{6}$ will satisfy the requests of the theorem with the original $V_k$. Therefore there is no loss of generality in assuming that $P\equiv 0$, i.e. $V_\infty'\le 0$. 

Taking $\lambda=\lambda_k\in (0,\frac{1}{24}]$ in Theorem \ref{fixed1} we have that for every $\Lambda>0$ there exists a solution $u_k\in C^5_{rad}(\R^6)$ to  
\begin{align*}
u_k(x)&=v_k-|x|^4+c_k\\
&=\frac{1}{\gamma_6}\int_{\R^6}\log\left(\frac{1}{|x-y|}\right)V_k(y)e^{-6|y|^4}e^{6(v_k(y)+c_k)}dy\\
&\quad +\lambda_k\D v_k(0)(|x|^4-2|x|^2)-|x|^4+c_{k}
\end{align*}
such that
\begin{align}\label{any-vol}
\Lambda=\int_{\R^6}V_k(y)e^{6u_k(y)}dy.
\end{align}
In particular $u_k$ solves the integral equation
\begin{equation}\label{exist}
u_k(x)=\frac{1}{\gamma_6}\int_{\R^6}\log\left(\frac{1}{|x-y|}\right)V_k(y)e^{6u_k(y)}dy+\lambda_k\D u_k(0)(1-|x|^2)^2-|x|^4+\tilde c_{k},
\end{equation}
where $$\tilde{c_k}:=c_k-\lambda_k\D u_k(0).$$
Computing the Laplacian at the origin on both sides of \eqref{exist} yields
\begin{equation}\label{Deltauneg0}
(1+24\lambda_k)\Delta u_k(0)=-\frac{4}{\gamma_6}\int_{\R^6}\frac{V_k(y)e^{6u_k(y)}}{|y|^2}dy,
\end{equation}
hence
\begin{equation}\label{Deltauneg}
\Delta u_k(0)<0.
\end{equation}

We will now prove Theorem \ref{thm-1} by fixing $\Lambda>\Lambda_1$, letting $\lambda_k\to 0^+$. 
The case $\Lambda=\Lambda_1$ can be easily deduced by first taking $\Lambda> \Lambda_1$ and then letting $\Lambda\to \Lambda_1^+$ slowly enough with a diagonal procedure.

The first step will be proving that $\lambda_k\Delta u_k(0)\to -\infty$ (so that $-\lambda_k\Delta u_k(0)$ plays the role of $\beta_k$ from Theorem \ref{ARSM} and Theorem \ref{thm-2}). A crucial tool will be the following Pohozaev-type identity, from \cite[Lemma 2.4]{Wei-Xu} (see also \cite[Theorem 2.1]{Xu}) 
for $w$ solving
$$w(x)=\frac{1}{\gamma_6}\int_{\R^6}\log\left(\frac{|y|}{|x-y|}\right)K(y)e^{6w(y)}dy+ c$$
with $|x\cdot\nabla K(x)|\leq C$, and
$$\alpha:=\int_{\R^6}Ke^{nw}dy,$$
we have
\begin{align*}
\frac{\alpha}{\gamma_6}\left(\frac{\alpha}{\gamma_6}-2\right)=\frac{1}{3\gamma_6}\int_{\R^6}\left(
y \cdot\nabla K(y)\right)e^{6w(y)}dy,
\end{align*}
which, observing that $2\gamma_6=\Lambda_1$, can be recast as
\begin{align}\label{nabla}
\frac{2\alpha}{\Lambda_1}\left(\alpha-\Lambda_1\right)=\frac{1}{3}\int_{\R^6}\left(
y \cdot\nabla K(y)\right)e^{6w(y)}dy.
\end{align}

\begin{lem}\label{vkproperty}
 Let $\Lambda>\Lambda_1$. Let $(u_k)$ be a sequence of radial solutions to \eqref{exist}-\eqref{any-vol} with $V_k$ as in Theorem \ref{thm-1}. Then  $\lambda_k\D u_k(0)\to-\infty$.
\end{lem}
\begin{proof} We proceed by steps. 

\noindent\textbf{Step 1} $u_k(x)\leq 16+u_k(2)-|x|^4$ for $|x|\geq 2$. 

Differentiating under the integral sign and observing that  $\D \log\frac{1}{|\cdot-y|}\leq 0$, from \eqref{exist}, we obtain $$\D \tilde u_k\leq 0,\quad \tilde u_k(x):=u_k(x)-\lambda_k\D u_k(0)(1-|x|^2)^2+|x|^4,$$ and by \eqref{w'} we have that $\tilde u_k$ is monotone decreasing. 
This proves Step 1, thanks to \eqref{Deltauneg}.  

\medskip 
\noindent\textbf{Step 2} For every $\delta>0$ we have $\sup_{B_\delta}u_k\to\infty$. 

Assume by contradiction that  $\sup_{k}\sup_{B_\delta}u_k<\infty$ for some $\delta>0$. Then by \eqref{any-vol}-\eqref{Deltauneg0} one has $|\D u_k(0)|\leq C$, which implies that  $\lambda_k\D u_k(0)\to0$. Then, by Theorem \ref{ARSM}, up to a subsequence either $(u_k)$ is bounded in $C^{5,\sigma}_{\loc}(\R^6)$ for $\sigma\in [0,1)$, or there exists $\beta_k\to\infty$ such that 
{ $$\frac{u_k}{\beta_k}\to\vp:=c_1+c_2r^2+c_3r^4\quad\text{in } C^5_{\loc}(\R^6),\quad  \vp\leq 0, \quad \vp\not\equiv0,$$
for some $c_1,c_2,c_3\in\R$. We claim that the latter case does not occur. Otherwise, differentiating under the integral sign in \eqref{exist}, one gets $\int_{B_r}|\D u_k|dx\leq C(r)$, hence $\D \vp\equiv 0$ in $\R^6$, that is, $\vp\equiv c_1<0$}. Then $u_k\to -\infty$ locally uniformly in $\R^6$, and by Step 1, $\int_{\R^6}V_ke^{6u_k}dx\to0$, a contradiction to \eqref{any-vol}. Thus, up to a subsequence, $u_k\to u_\infty$ in $C^5_{\loc}(\R^6)$.  We claim that  $u_\infty$ satisfies 
$$u_\infty(x)=\frac{1}{\gamma_6}\int_{\R^6}\log\left(\frac{1}{|x-y|}\right)V_\infty(y)e^{6u_\infty(y)}dy-|x|^4+c=:\bar u_\infty(x),$$
with 
$$\alpha:=\int_{\R^6}V_\infty e^{6u_\infty}dx=\Lambda,\quad c:=\lim_{k\to\infty}\tilde c_k.$$
It follows from Step 1 that $u_\infty(x)\leq C-|x|^4$ on $\R^6$. Using this one can show that $u_k\to \bar u_\infty$ in $C^0_{loc}(\R^6)$, and hence $u_\infty=\bar u_\infty$. 

To show that $\alpha=\Lambda$ we use Step 1. Indeed, as  $u_k(2)\leq C$
$$\alpha=\lim_{L\to\infty}\lim_{k\to\infty}\int_{B_L}V_ke^{6u_k}dx=\Lambda-\lim_{L\to\infty}\lim_{k\to\infty}\int_{B_L^c}V_ke^{6u_k}dx=\Lambda.$$
Since $V_\infty'\leq 0$,  applying  \eqref{nabla} with $w=u_\infty+|x|^4$ and $K=V_\infty e^{-|x|^4}$ one gets $\alpha<\Lambda_1$,  a contradiction as   $\alpha>\Lambda_1$. 



\medskip 
\noindent\textbf{Step 3} $\lambda_k\D u_k(0)\to-\infty$

We assume by contradiction that $\lambda_k \D u_k(0)\geq -C$ for some $C>0$. Then, differentiating in \eqref{exist} and using Fubini's theorem, we get for every $R_0>0$
$$\int_{B_{2R_0}}|\nabla u_k(x)|dx\leq C(R_0).$$
Therefore, by \cite[Theorem 1]{Mar2} (see also \cite{Rob2}), as $u_k$ is radially symmetric
$$\int_{B_{R_0}}V_ke^{6u_k}dx\to \Lambda_1.$$
Set now $\ve_0=\frac12(\Lambda-\Lambda_1)$ and fix $R_0=R_0(\ve,\Lambda)$  such that $$\int_{B_{R_0}^c}V_ke^{6u_k}dx\leq C\int_{B_{R_0}^c}e^{-|x|^4}dx<\ve_0.$$ Then we obtain $$\Lambda=\int_{B_{R_0}}V_ke^{6u_k}dx+\int_{B_{R_0}^c}V_ke^{6u_k}dx<\frac{\Lambda+\Lambda_1}{2}+o(1),$$
which is a contradiction.
\end{proof}

\begin{lem}\label{lemetak} Under the same assumptions of Lemma \ref{vkproperty}, we have that $u_k(0)\to\infty$ and \eqref{etaklog}-\eqref{limcurv} hold with $r_k=2e^{-u_k(0)}$ and any $\delta\in (0,1)$.
\end{lem}

\begin{proof} We proceed by steps.

\noindent \textbf{Step 1} There exists a radius $\xi_k\to 1$ such that $\xi_k$ is a local maxima of $u_k$. Indeed differentiating \eqref{exist} under the integral sign we obtain 
 \begin{align}
  u_k'(r)&=\frac{1}{\omega_5r^5}\int_{B_r}\D u_k(x)dx\notag\\
  &=C\frac{1}{r^5}\int_{\R^6}V_k(y)e^{6u_k(y)}\int_{B_r}\frac{dx}{|x-y|^2}dy-\frac{1}{\omega_5r^5}\int_{B_r}\D \left(-\lambda_k\D u_k(0)(1-|x|^2)^2+|x|^4\right)dx\notag\\
  &=\frac{O(1)}{r}-4\lambda_k \D u_k(0)r(1-r^2)-4r^3. \label{O(1)r}
 \end{align}
Then, using Lemma \ref{vkproperty} we infer
$$u_k'\xrightarrow{k\to\infty} +\infty\quad \text{uniformly locally in } (0,1)$$
and
$$u_k'\xrightarrow{k\to\infty} -\infty\quad \text{uniformly locally in } (1,\infty).$$
This proves the claim. 

\medskip

\noindent \textbf{Step 2} We claim that $u_k(0)\to\infty$. Indeed  a simple application of \eqref{w'}, together with $\Delta^3 u_k<0$ implies that $u_k$ can have at most two local maxima (compare to the proof of Lemma \ref{monotone}). From \eqref{Deltauneg} and the previous step we infer that $0$ and $\xi_k$ are these local maxima. Since $\xi_k\to 1$, the claim now follows at once from Step 2 of Lemma \ref{vkproperty}. 
 
\medskip


\noindent \textbf{Step 3} We claim that $u_k$ satisfies \eqref{etaklog}. Indeed, as $u_k(0)\to\infty$, $u_k(0)$ is the global maximum of $u_k$ on $[0,\frac12]$ and $u_k\to -\infty$ locally uniformly in $(0,\frac12]$. Then we can apply Lemma \ref{lemetaklog} with $\alpha_k=-\lambda_k \Delta u_k(0)\to +\infty$ (by Lemma \ref{vkproperty}), $C_1=24-32\tau^2$, $C_2=-384$ and $C_3=24$ for some $\tau\in (0,\tfrac12]$ to obtain that \eqref{etaklog}.

\medskip

\noindent \textbf{Step 4} \eqref{limcurv} hold for every $\delta\in (0,1)$.

Let us verify that the assumptions of Lemma \ref{neck1} are satisfied for any fixed $\delta\in (0,1)$.
From Step 1 and Step 2 we can find $\theta_k\in (0,1)$ such that \eqref{-uk'} holds, while the assumptions of Lemma \ref{lemetaklog} have already been verified. Moreover $u_k(\theta_k)\le u_k(\delta)\to-\infty$. Then \eqref{limcurv} follows from Lemma \ref{neck1}.
\end{proof}

\medskip\noindent\textbf{Proof of Theorem \ref{thm-1} (completed).}
Taking into account Lemma \ref{lemetak},   if we show that $u_k(1)\to\infty$, then we are in case $iv)$ of Theorem \ref{thm-2} with $\rho=1$.

From \eqref{exist} we bound 
\begin{align*}
 u_k(1)&\geq\frac{1}{\gamma_6}\int_{|e_1-y|>1}\log\left(\frac{1}{|e_1-y|}\right)V_k(y)e^{6u_k(y)}dy -1+\tilde c_k,
\end{align*}
where $e_1=(1,0,\dots,0)$ is a unit vector. By Step 1 of Lemma \ref{vkproperty}
$$\int_{|e_1-y|>1}\log\left(\frac{1}{|e_1-y|}\right)V_k(y)e^{6u_k(y)}dy=O(1).$$
We claim that $\tilde c_k\to\infty$. In order to prove the claim we set $$\mathcal{C}_r:=\{x=(x_1,\bar x)\in\R^6:|\bar x|\leq rx_1\}.$$
Since $u_k$ is radial and satisfies \eqref{any-vol}, we can choose $r_0>0$ small such that
$$\int_{\C_{2r_0}}V_ke^{6u_k}dx\leq\frac{\Lambda_1}{4}.$$
Hence, for $x\in\C_{r_0}\setminus B_\frac12$
$$u_k(x)\leq C+\frac{1}{\gamma_6}\int_{\C_{2r_0}}\log\left(\frac{1}{|x-y|}\right)V_k(y)e^{6u_k(y)}dy+\lambda_k\D u_k(0)(1-|x|^2)^2+\tilde c_k.$$
Then by Jensen's inequality and Fubini's theorem
\begin{align*}
\int_{\C_{r_0}\setminus B_\frac12}V_k(x)e^{6u_k(x)}dx&\leq Ce^{6\tilde c_k}\int_{\C_{2r_0}}\frac{f_k(y)}{\|f_k\|}\int_{C_{r_0}}V_k(x) \frac{e^{6\lambda_k\D u_k(0)(1-|x|^2)^2}}{|x-y|^p}dxdy\\
&=e^{6 \tilde c_k}o(1),
\end{align*}
thanks to Lemma \ref{vkproperty}, where $f_k:=V_ke^{6u_k}$, $\|f_k\|:=\|f_k\|_{L^1(\C_{2r_0})}$ and $p:=\frac{6}{\gamma_6}\|f_k\|<6$. Observing that by Lemma \ref{lemetak}
\begin{align*}
0&<\Lambda-\Lambda_1
=\int_{\R^6\setminus B_{\frac12}} V_ke^{6u_k}+o(1)\leq C(r_0)\int_{C_{r_0}\setminus B_\frac12}V_ke^{6u_k}dx+o(1),
\end{align*}
we conclude our claim. 

This proves that $u_k(1)\to \infty$, hence we are in case $iv)$ of Theorem \ref{thm-2}. 
Then from \eqref{ukbetak} and Step 1 of Lemma \ref{vkproperty} it follows 
$$\int_{B_{1+\ve}\setminus B_{1-\ve}}V_ke^{6u_k}dx=\Lambda-\Lambda_1+o(1)$$
for every $\ve\in (0,1)$.
\hfill $\square$



\section{{The case of dimension $2n\ge 6$}}\label{dim2n}

Similar to Theorem \ref{thm-1} one can prove the following.

\begin{thm}\label{thm-1bis} Let $(V_k)\subset C^0_{rad}(\R^{2n}) $ be positive functions with $V_k\to V_\infty$  uniformly, where $V_\infty\in C^1(\R^{2n})$, $V_\infty(0)=(2n-1)!$ and for some $a,b\ge 0$
 \begin{align}\label{condVinfty}
 \frac{d}{dr}\bra{\frac{V_\infty (r)}{e^{ar^2+br^4}}}\leq 0, \quad r\in [0,\infty).
 \end{align}
Then  
for every $\Lambda \ge\Lambda_1$ there exists a sequence $(u_k)\subset C^{2n}_{rad}(\R^{2n})$ of solutions to \eqref{eq-1}-\eqref{eq-2} with $\Omega=\R^{2n}$ such that $u_k(0)\to\infty$ and $u_k(1)\to \infty$ as $k\to\infty$, and for every $\ve\in (0,1)$
\begin{align}
&\lim_{k\to\infty}\int_{B_{1+\ve}\setminus B_{1-\ve}}V_ke^{2nu_k}dx=\Lambda-\Lambda_1,\label{curv1bis}\\
&\lim_{k\to\infty}\int_{\R^{2n}\setminus (B_\ve\cup (B_{1+\ve}\setminus B_{1-\ve})}V_ke^{2nu_k}dx=0.\label{curv2bis}
\end{align}
\end{thm}

\begin{proof}
Again using the existence result of \cite{H-volume}, for $n\geq 3$, $\lambda_k\in (0,\frac{1}{8n})$ and $\Lambda>\Lambda_1=(2n-1)!|S^{2n}|$ we find a solution to 
\begin{equation}\label{exist-2m}
u_k(x)=\frac{1}{\gamma_{2n}}\int_{\R^{2m}}\log\left(\frac{1}{|x-y|}\right)V_k(y)e^{2mu_k(y)}dy+\lambda_k\D u_k(0)(1-|x|^2)^2-|x|^4+\tilde c_{k},
\end{equation} such that $\D u_k(0)<0$ and $$\int_{\R^{2n}}V_ke^{2nu_k}dx=\Lambda.$$
Differentiating under the integral sign, and using that 
$$\D^3 \log\frac{1}{|x|}=-\frac{8(m-2)[(m-5)^2-1]}{|x|^6}\quad \text{in }\R^m,$$ we see that
$$\D^3 u_k<0\quad\text{in }\R^{2n}.$$

{We now send $\lambda_k\to 0^+$} and want to show that $u_k(0)\to\infty$ and $u_k(1)\to\infty$. This can be done in the following steps.

\medskip

\noindent \textbf{Step 1} $\lambda_k\D u_k(0)\to-\infty$. In particular, $\sup_{B_\delta(0)} u_k\to\infty$ for every $\delta>0$. This can be proven with the same argument of Lemma \ref{vkproperty}. 

\medskip

\noindent \textbf{Step 2} There exists $\beta_k\to\infty$ such that $$\frac{u_k}{\beta_k}\to\vp\quad\text{in }C^{2n-1}_{loc}({\R^{2n}}\setminus S),\quad \vp=-c_1(1-|x|^2)^2+c_2,$$ for some $c_1\geq 0,\, c_2\leq 0$, where $S=\{0\}\cup S_\vp$. 

From \eqref{exist-2m} we see that $\vp$ should be of the form $\vp=-c_1(1-|x|^2)^2+c_2$ for some $c_1,c_2\in\R$.  Since $\vp\leq 0$ on $\R^{2n}$, we get $c_1\geq 0$ and $c_2\leq 0$.

\medskip

\noindent \textbf{Step 3} (Monotonicity) $u_k$ has two local maximum points, namely $0$ and a point  $\xi_k\to1$.  Indeed, since  $\D u_k(0)<0$, $0$ is a local maxima and the existence of $\xi_k$ follows as in Step 1 of Lemma \ref{lemetak}.  
To show that $u_k$ can not have another point of local maxima we need to use that $\D^3 u_k<0$ in {$\R^{2n}$}. First we show that  the same conclusion of Lemma 2.1 holds. 
We can repeat the same proof simply replacing \eqref{vpdelta} by $$u_k(1\pm\delta)\to-\infty,\quad u_k'(1-\delta)\to\infty,\quad u_k'(1+\delta)\to-\infty,$$ which follows from  \eqref{exist-2m}.

\medskip

\noindent \textbf{Step 4} $u_k(0)\to\infty$ and $u_k(0)=\sup_{B_\frac12}u_k$. This follows trivially from the above steps. 

\medskip

\noindent \textbf{Step 5} Blow-up at the origin is spherical.  This can be proven as in Lemma \ref{lemetaklog}.  

\medskip

\noindent \textbf{Step 6} There is concentration at the origin. This can be done as in subsection 2.3. 

\medskip

\noindent \textbf{Step 7} $u_k(1)\to\infty$. It suffices to show that the constant (appearing in \eqref{exist-2m}) $\tilde c_k\to\infty$. The proof is exactly as the case of dimension $6$. 
\end{proof}

\section{Proof of Theorem \ref{thm-exp}}

\subsection{Proof of \eqref{ukglobal} and \eqref{betakuk}}

We will now establish some relations among $\theta_{i,k}$, $\beta_k$ and $u_k(0)$ that will lead to the proofs of \eqref{ukglobal} and \eqref{betakuk}. We start with a preliminary lemma.
 
 \begin{lem}\label{2.6}
 For every $0<\xi_k<\tilde\xi_k$ with $r_k=o(\xi_k)$ and $\tilde\xi_k=o(\theta_{4,k})$ we have  $$\D^2u_k(x)=\frac{32+o(1)}{|x|^4}\quad \text{on }  B_{\tilde\xi_k}\setminus B_{\xi_k}.$$ 
In particular, $$\int_{B_t}\D^2u_k(x)dx=16\omega_5(1+o(1))t^2\quad\text{for }\xi_k\leq t\leq \tilde\xi_k.$$
  \end{lem}
 \begin{proof}
 Using \eqref{etaklog}, \eqref{limcurv} and recalling that $\frac{\Lambda_1}{\omega_5}=128$ we obtain for  $\xi_k\leq |x|\leq\tilde\xi_k$
 \begin{align*}
 \D^2u_k(x)&=\int_{|x|}^{\theta_{4,k}}\frac{1}{\omega_5t^5}\int_{B_t}V_k(y)e^{6u_k(y)}dydt\\
 &=(\Lambda_1+o(1))\int_{|x|}^{\theta_{4,k}}\frac{1}{\omega_5t^5}dt\\
 &=(32+o(1))\frac{1}{|x|^4}.
 \end{align*}
 The second part  follows from the first part and Lemma \ref{stimeDelta}.
 \end{proof}

From the definition of $\theta_{i,k}$ one has $\D^\frac i2u_k(\theta_{i,k})=0$ for $i=2,4$ and $(\D^{\frac{i-1}{2}}u_k)'(\theta_{i,k})=0$ for $i=1,3$. 

  \begin{lem} \label{akbk}
 We have  
  \begin{itemize}
  \item [i)] $\theta_{i,k}e^{u_k(0)}\to\infty$ for  $i=0,1,2,3,4$.
   \item [ii)] $\beta_k\theta_{4,k}^4\to\frac{1}{12}$. 
   \item [iii)] $\beta_{k}\theta_{2,k}^2\to\frac13.  $  
   
   \item [iv)]   $\theta_{3,k}\approx \theta_{4,k}$
   
   \item [v)]   $\theta_{1,k}\approx \theta_{2,k}$ 
        \end{itemize}
  \end{lem}
\begin{proof}  $i)$ follows from the definition of $\theta_{i,k}$ and \eqref{etaklog}.

\medskip
Since $\Delta^2 (1-r^2)^2= 384$, by \eqref{ukbetak} and \eqref{limcurv} we have
 \begin{align}
 0&=\Delta^2u_k(\theta_{4,k})=\Delta^2u_k\left(\tfrac12\right)+\int_{\theta_{4,k}}^\frac12\frac{1}{\omega_5r^5}\int_{B_r}V_ke^{6u_k}dxdr\notag\\
 &= (-384+o(1))\beta_k+(\Lambda_1+o(1))\int_{\theta_{4,k}}^\frac12\frac{1}{\omega_5r^5}dr\notag\\
 &=(-384+o(1))\beta_k+(\Lambda_1+o(1))\frac{1}{4\omega_5\theta_{4,k}^4}+C,  \notag
\end{align}
hence
$$\lim_{k\to\infty}\beta_k\theta_{4,k}^4=\frac{\Lambda_1}{1536\omega_5}=\frac{1}{12},$$
and this proves $ii)$. 

\medskip 

To prove  $iii)$ fix $\ve>0$ arbitrarily small. Then by Lemma \ref{stimeDelta} and \eqref{ukbetak} we get 
\begin{align}
(24-32\ve^2+o(1))\beta_k&=\D u_k(\ve)-\D u_k(\theta_{2,k})\notag\\
&=\int_{\theta_{2,k}}^\ve\frac{1}{\omega_5t^5}\int_{B_t}\D^2u_k(x)dxdt \label{17} \\ 
&\leq \frac{C}{\theta_{2,k}^2}.\notag
\end{align}
This shows that $\beta_k\theta_{2,k}^2\leq C$, and in particular  $\theta_{2,k}=o(\theta_{4,k})$, thanks to $ii)$. Therefore, we can choose $\xi_k\in (\theta_{2,k},\theta_{4,k})$ such that $\theta_{2,k}=o(\xi_k)$ and $\xi_k=o(\theta_{4,k})$. 

We write the integral in \eqref{17} as $I_1+I_2$ where  $$I_1:=\int_{\theta_{2,k}}^{\xi_k}\frac{1}{\omega_5t^5}\int_{B_t}\D^2u_k(x)dxdt ,\quad I_2:=\int_{\xi_k}^\ve\frac{1}{\omega_5t^5}\int_{B_t}\D^2u_k(x)dxdt. $$
Using Lemma \ref{2.6} we compute
\begin{align*}
I_1=16(1+o(1))\int_{\theta_{2,k}}^{\xi_k}\frac{1}{t^3}dt =\frac{8+o(1)}{\theta_{2,k}^2}.
\end{align*}
Using that $-C\beta_k\leq \D^2u_k(x)\leq \frac{C}{|x|^4}$ we bound $$|I_2|\leq C\ve^2\beta_k+\frac{C}{\xi_k^2}=C\ve^2\beta_k+\frac{o(1)}{\theta_{2,k}^2}.$$
Now $iii)$ follows  from \eqref{17} as $\ve>0$ is arbitrary and $\theta_{2,k}^2\beta_k\le C$. 

\medskip

We now prove $iv)$. From Lemma \ref{monotone} we have $\theta_{4,k}<\theta_{3,k}$. Taking $r=\theta_{3,k}$ in $ii)$ of Lemma \ref{derivative-uk} we have $\theta_{3,k}^4\beta_k\leq C$, and hence $\theta_{3,k}=O(\theta_{4,k})$, thanks to $ii)$. 

\medskip
Finally, we prove $v)$. From Lemma \ref{monotone} we have $\theta_{2,k}<\theta_{1,k}$. Taking $r=\theta_{1,k}$ in $iii)$ of Lemma \ref{derivative-uk} we have $\theta_{1,k}^2\beta_k\leq C$, and hence $\theta_{1,k}=O(\theta_{2,k})$, thanks to $iii)$. 
\end{proof}


 \begin{lem}\label{max-laplacian}
 We have
$$\lim_{k\to\infty}\frac{\D u_k(\theta_{3,k})}{\beta_k}=\D\vp(0)=24.$$
 \end{lem}
\begin{proof}
Since $\D \vp(\frac{\sqrt{3}}{2})=0$, there exists a sequence $\tilde\theta_{2,k}\to\frac{\sqrt{3}}{2}$ such that $\D u_k(\tilde\theta_{2,k})=0$.  Hence, by Lemma \ref{monotone} we infer
$$0\le -\Delta^2 u_k(r)\le -\Delta^2 u_k(\tilde\theta_{2,k}),\quad \text{for }\theta_{3,k}\le r\le \tilde\theta_{2,k},$$
and with \eqref{ukbetak} we get 
\begin{align}
 \D u_k(\theta_{3,k})&=\int_{\theta_{3,k}}^{\tilde\theta_{2,k}}\frac{1}{\omega_5t^5}\int_{B_t}(-\D^2u_k(x))dxdt\notag
 \leq -\D^2 u_k(\tilde\theta_{2,k})\frac{\tilde\theta_{2,k}^2-\theta_{3,k}^2}{12}\notag
 =(24+o(1))\beta_k.\notag
\end{align}
The lemma follows immediately by \eqref{ukbetak}.
\end{proof}


\paragraph{Proof of \eqref{ukglobal} and \eqref{betakuk} (completed).}
According to Lemma \ref{akbk} we have $\theta_{2,k}=o(\theta_{4,k})$, hence we can choose $s_k\in(\theta_{2,k},\theta_{4,k})$  such that $\theta_{2,k}=o(s_k)$ and $s_k=o(\theta_{4,k})$. We claim that
$$\D^2u_k(r)=(1+o(1))\D^2\bar\eta_k(r),\quad 0\leq r\leq s_k.$$ 
Indeed assume by contradiction that there exists $\mu_k\in [0,s_k]$ such that $$|\D^2 u_k(\mu_k)-\D^2\bar\eta_k(\mu_k)|\geq \ve |\D^2\bar\eta_k(\mu_k)|\quad\text{for some }\ve>0.$$ It follows from \eqref{betakuk(0)} that $r_k=o(\mu_k)$.  Therefore, by Lemma \ref{2.6} $$\D^2 u_k(\mu_k)=\frac{32+o(1)}{\mu_k^4},$$
and since
$$\D^2 \bar\eta_k(\mu_k)=\frac{32+o(1)}{\mu_k^4},$$
we get a contradiction.
Therefore, for $0\leq r\leq s_k$
  \begin{align*}
  \D u_k(r)&=\int^{r}_{\theta_{2,k}}\frac{1}{\omega_5t^5}\int_{B_t}\D^2u_k(y)dydt\\
  &=\int^{r}_{\theta_{2,k}}\frac{1}{\omega_5t^5}\int_{B_t}(1+o(1))\D^2\bar\eta_k(y)dydt\\
  &=(1+o(1))(\D\bar\eta_k(r)-\D\bar\eta_k(\theta_{2,k}))\\ 
  &=(1+o(1))(\D\bar\eta_k(r)+24\beta_k) +o(\beta_k),
  \end{align*}
  where the last equality follows from $iii)$ of Lemma \ref{akbk} and  $\D\bar\eta_k(\theta_{2,k})\theta_{2,k}^2\to -8$. 
  
 We now claim that
$$  \D u_k(x)=(1+o(1))(\D\bar \eta_k(x) +\beta_k\D\vp(x))+o(\beta_k) \quad\text{on any compact set }K\Subset B_R.$$  
Indeed assume by contradiction that  for some $\mu_k\in K$ $$|\D u_k(\mu_k)-\D \bar\eta_k(\mu_k)-\beta_k\D\vp(\mu_k)|\geq \ve(|\D\bar\eta_k(\mu_k)|+\beta_k).$$  From the first part we have that $\mu_k\geq s_k.$  As $\beta_k\approx \theta_{2,k}^{-2}$ and $\theta_{2,k}=o(s_k)$, we must have $\D \bar\eta_k(\mu_k)=o(\beta_k)$. By \eqref{ukbetak} we get $\mu_k\to0$. Hence $$|\D u_k(\mu_k)-\beta_k\D\vp(\mu_k)|\geq \frac{\ve}{2}\beta_k.$$ It follows from Lemma \ref{max-laplacian} that $$\D u_k(\mu_k)\leq 24\beta_k-\frac{\ve}{4}\beta_k,$$ and in fact, $\mu_k<\theta_{3,k}$, thanks to the monotonicity of $\D u_k$ and \eqref{ukbetak}. Using that {$\theta_{2,k}=o(s_k)=o(\mu_k)$, and recalling that $\beta_k\approx \theta_{2,k}^{-2}$ and } $\D^2u_k(x)\leq\frac{C}{|x|^4}$ we obtain 
$$\frac{\ve+o(1)}{4}\beta_k\leq \D u_k(\theta_{3,k})-\D u_k(\mu_k)=\int_{{\mu_k}}^{\theta_{3,k}}\frac{1}{\omega_5t^5}\int_{B_t}\D^2 u_k(x)dxdt\leq \frac{C}{\mu_k^2}=o(\beta_k),$$ a contradiction.

Therefore given $\delta\in (0,R)$ fixed, and using $u_k(0)=\bar\eta_k(0)$, we get for  $r\in (0,\delta)$
\begin{align}
u_k(r)&=u_k(0)+\int^{r}_{0}\frac{1}{\omega_5t^5}\int_{B_t} (1+o(1))(\D\bar \eta_k(x) +\beta_k\D\vp(x))dxdt+o(\beta_k) \notag\\
&=u_k(0)+(1+o(1))(\bar\eta_k(r)-\bar\eta_k(0)+\beta_k(\vp(r)-\vp(0))) +o(\beta_k+u_k(0)) \notag\\
&=\bar\eta_k(r)+\beta_k(\vp(r)+1)+o(\beta_k+u_k(0)).\label{uk}
\end{align}
By \eqref{ukbetak} we have $u_k(\tfrac12)=(1+o(1))\beta_k\vp(\frac12)$, and hence, from \eqref{uk} we infer
$$\beta_k+\bar\eta_k(\tfrac12) +o(\beta_k+u_k(0))=0 .$$
Since  $\bar\eta_k(\tfrac12)=(-1+o(1))u_k(0)$, \eqref{betakuk} follows at once. Then \eqref{ukglobal} follows from \eqref{uk}.
\hfill $\square$

\subsection{Linearization and proof of \eqref{etakexp}-\eqref{stimapsi}}

Let $\eta_k$ and $\eta$ be as in \eqref{etaklog}. 
We set 
\begin{align}\label{defpsik}
\psi_k(x):=\frac{1}{\tilde\ve_k}\left(\eta_k(x)-\eta(x)\right),
\end{align}
where {(notice that $r_k^2\sqrt{\beta_k}\to 0$ by \eqref{betakuk(0)})}
\begin{align}\label{epsilonk}
\tilde\ve_k:=\max\left\{|\D \eta(0)-\D \eta_k(0)|,\,|\D^2 \eta(0)-\D^2 \eta_k(0)|,r_k^2\sqrt{\beta_k}\right\}\xrightarrow{k\to\infty}0.
\end{align}
We will show later that $\tilde \ve_k\approx r_k^2\beta_k\approx \ve_k:=u_k(0)e^{-2u_k(0)}$. 
For any $R_0>0$ we have
$$|\tilde \ve_k\psi_k(x)|=|\eta_k(x)-\eta(x)|=o(1)\quad \text{on }B_{R_0}.$$
Therefore for $x\in B_{R_0}$, using \eqref{assVk}, so that
$$\tilde V_k(x):=V_k(r_kx)=120+O(r_k^2),$$
we compute with a Taylor expansion
 \begin{align*}
(-\D)^3\psi_k&=\frac{e^{6\eta}}{\tilde\ve_k}\left[\tilde V_ke^{6\ve_k\psi_k}-120\right]\\
&=\frac{e^{6\eta}}{\tilde \ve_k}\left[(120+O(r_k^2))(1+6\tilde\ve_k\psi_k+o(\ve_k\psi_k))-120\right]\\
&={720}e^{6\eta}\psi_k(1+o(1))+o(1),
\end{align*}
where we also used that $O(\tilde\ve_k^{-1}r_k^2)=o(1)$.
Since
$$\psi_k(0)=0,\quad |\Delta \psi_k(0)|\le 1,\quad |\Delta^2\psi_k(0)|\le 1, $$
by ODE theory $\psi_k$ converges up to a subsequence 
to $\psi$ in $C^5_{\loc}(\R^6)$ where $\psi$ is a radial solution to
\begin{equation}\label{eqpsi}
(-\D)^3\psi=720\psi e^{6\eta}\quad\text{in }\R^6,
\end{equation}
with $\psi(0)=0$.

The following proposition collects some crucial properties about the solutions to \eqref{eqpsi}. We shall prove it in Section \ref{proofpsi}.

\begin{prop}\label{psi-vol}
Let $\psi$ be a radial solution to \eqref{eqpsi}. Then 
\begin{equation}\label{psiasym}
\psi(x)=P(|x|)-\alpha\log|x|+o(\log|x|),\quad \text{as } |x|\to\infty,
\end{equation}
where $P(r)=ar^2+br^4+d$ for some $a,b,d\in\R$, $o(\log|x|)$ satisfies
\begin{equation}\label{psiasymbis}
\nabla^j o(\log|x|)=o(|x|^{-j}) \quad \text{as }|x|\to\infty,\quad \text{for }1\le j\le5,
\end{equation}
and
$$\alpha=\frac{720}{\gamma_6}\int_{\R^6}\psi(y)e^{6\eta(y)}dy=6a+48b.$$
Finally, if $a=b=0$, then $\psi(r)=\gamma\frac{1-r^2}{1+r^2}$ for some $\gamma\in \R$.
\end{prop}


\begin{rem}\label{rem1}
Notice that $\psi\equiv 0$ if and only if   $\tilde\ve_k=r_k^2\sqrt{\beta_k}$ for $k$ large and $$\frac{1}{\tilde\ve_k}(|\D \eta(0)-\D \eta_k(0)|+|\D^2 \eta(0)-\D^2 \eta_k(0)|)\xrightarrow{k\to\infty}0.$$ 
\end{rem}

We now write 
\begin{align}
 \eta_k=\eta+\tilde\ve_k\psi+\tilde\ve_k\delta_k\phi_k,\label{defphik}
\end{align}
where 
\begin{align}\label{deltak}
       \delta_k:=\max\left\{\tilde\ve_k,\,\frac{1}{\sqrt{\beta_k}},\,|\D\psi_k(0)-\D\psi(0)|,\,|\D^2\psi_k(0)-\D^2\psi_k(0)|\right\}\xrightarrow{k\to\infty}0.
 \end{align}
Then
\begin{align*}
 (-\D)^3\phi_k&=\frac{e^{6\eta}}{\tilde\ve_k\delta_k}\left(\tilde V_ke^{6(\tilde\ve_k\psi+\tilde\ve_k\delta_k\phi_k)}-120-720\tilde\ve_k\psi\right) =:\Phi_k(\phi_k).
 \end{align*}
On any fixed ball $B_{R_0}\subset\R^6$ we have  $\tilde\ve_k\psi=o(1)$, $\tilde\ve_k\delta_k\phi_k=o(1)$, hence with a Taylor expansion we get
\begin{align}
\Phi_k(\phi_k)&= \frac{120e^{6\eta}}{\tilde\ve_k\delta_k}\left[  (1+O(r_k^2R_0^2))(1+6\tilde\ve_k\psi+6\tilde\ve_k\delta_k\phi_k+O((\tilde\ve_k\psi+\tilde\ve_k\delta_k\phi_k)^2))-1-6\tilde\ve_k\psi \right]\notag\\
&=720e^{6\eta}\left(\phi_k+o(\phi_k)+O(\psi^2) +O(R_0^2)\right). \label{PhiTaylor}
\end{align}
Then, since
$$\phi_k(0)=0,\quad |\Delta \phi_k(0)|\le 1,\quad |\Delta^2\phi_k(0)|\le 1, $$
by ODE theory the sequence $(\phi_k)$ is bounded in $C^5_{\loc}(\R^6)$. 

We now bound $\phi_k$ on large scales.

\begin{lem}\label{2.8}
 Let $s_k>0$ be such that $s_k:=o(1) \tilde\ve_k^{-\frac14}.$ Then 
 $$\sup_{r\in [0,s_k]}\left(|\D^2\phi_k(r)|+|\D\phi_k(r)|(1+r)^{-2}+|\phi_k(r)|(1+r)^{-4}\right)\le C, \quad\forall \,k.$$
\end{lem}
\begin{proof}
Let $R_0>1$ to be fixed later.  We set 
\begin{align*}
X&:=\left\{\phi\in C^4([R_0,s_k]): \|\phi\|_X<\infty\right\},\\
\|\phi\|_X&:=\sup_{r\in [R_0,s_k]}\left(|\D^2\phi(r)|+|\D\phi(r)|r^{-2}+|\phi(r)|r^{-4}\right).
\end{align*}
Thanks to Proposition \ref{psi-vol} we have
$$\sup_{[R_0,s_k]}\tilde\ve_k\psi=o(1).$$
Moreover given a constant $M\ge 1$ to be fixed later, we set 
$$\mathcal{B}_M:=\{\phi\in X:\|\phi\|_X\le M\}.$$
Then we have
$$\sup_{[R_0,s_k]} \tilde\ve_k\delta_k\phi=o(1)\quad \text{for }\phi\in \mathcal{B}_M.$$
Therefore the same Taylor expansion used in \eqref{PhiTaylor} leads to
$$\Phi_k(\phi)(r)=720e^{6\eta(r)}\left(\phi(r)+o(\phi)+O(\psi^2(r)) +O(r^2)\right)\quad \text{for }\phi\in \mathcal{B}_M.$$
We now fix $k$ and define $T:X\to X$, $\phi\mapsto\bar\phi$ where $\bar\phi$ is the unique solution to the ODE
 $$(-\D)^3\bar\phi=\Phi_k(\phi)\quad\text{on }(R_0,s_k),\quad \bar\phi^{(j)}(R_0)=\phi_k^{(j)}(R_0),\quad j=0,1,\dots,5.$$ 
Using that
\begin{align*}
|\phi(r)+o(\phi(r))+O(\psi^2(r)) +O(r^2)|&\leq C(\|\phi\|_X r^4+r^8)\quad\text{ on }(R_0,s_k)
\end{align*}
we infer
$$|\Phi_k(\phi)(r)|= \frac{O(M)}{(1+r)^8}+\frac{O(1)}{(1+r)^4}\quad \text{for }r\in [R_0,s_k],\;\phi\in\mathcal{B}_M.$$
Then also using
$$|\phi_k^{(j)}(R_0)|\le C(R_0) \quad \text{ for } 0\le j\le 5,$$
and \eqref{int-repre} we bound uniformly for $r\in [R_0,s_k]$ and $\phi\in \mathcal{B}_M$
\begin{align*}
|\Delta^2\bar\phi(r)|&\le |\Delta^2 \bar\phi(R_0)|+{ R_0^5(\D^2\bar\phi)'(R_0)\int_{R_0}^r\frac{d\rho}{\rho^5}} +\int_{R_0}^r\frac{1}{\omega_5\rho^5}\int_{B_\rho{\setminus B_{R_0}}}|\Phi_k(\phi)(x)|dxd\rho\\
&\le C(R_0)+\int_{R_0}^r\frac{1}{\omega_5\rho^5}\int_{B_\rho}{\bra{\frac{O(M)}{(1+|x|)^8}+\frac{O(1)}{(1+|x|)^4}} }dxd\rho\\&\le C_1(R_0)+O(MR_0^{-4}).
\end{align*}
Similarly
\begin{align*}
|\Delta\bar\phi(r)|&\le |\Delta \bar\phi(R_0)|+ R_0^5(\D\bar\phi)'(R_0)\int_{R_0}^r\frac{d\rho}{\rho^5}+\int_{R_0}^r\frac{1}{\omega_5\rho^5}\int_{B_\rho\setminus B_{R_0}}|C_1(R_0)+O(MR_0^{-4})|dxd\rho\\
&\le (C_2(R_0)+O(MR_0^{-4}))r^2
\end{align*}
and integrating once more
$$|\bar\phi(r)|\le (C_3(R_0)+O(MR_0^{-4}))r^4.$$
Therefore,
 $$\|\bar\phi\|_X\leq C_4(R_0)+C_5MR_0^{-4}\quad \text{for }\phi\in \mathcal{B}_M.$$
Now we fix  $R_0>1$ so that $|C_5R_0^{-4}|\leq \frac12$. Then for $M\ge 2C_4(R_0)$ one has 
 $$\|\bar\phi\|_X\leq \frac M2+\frac{M}{2}\leq M,\quad \text{for every }\phi\in\mathcal{B}_M,$$
 i.e. $T$ sends the convex set $\mathcal{B}_M$ into itself. Then, by the Schauder fixed-point theorem (notice that $T$ is compact, as one gets easily bound on the fifth order derivative of $ \bar\phi$), $T$ has a fixed point   $\phi_*$ in $X$ with $\|\phi_*\|_X\le M$, that is, $\phi_*$ satisfies 
  $$(-\D)^3\phi_*=\Phi_k(\phi_*)\quad\text{on }(R_0,s_k),\quad \phi_*^{(j)}(R_0)=\phi_k^{(j)}(R_0),\quad j=0,1,\dots,5.$$ Therefore, from the uniqueness of solution $\phi_*=\phi_k|_{[R_0,s_k)}$, and  the lemma follows from the estimate $\|\phi_k\|_X\le M$. 
\end{proof}

\begin{lem}\label{vol-psik} 
We have
\begin{equation}\label{stimeinteg}
\int_{B_r}\left(\tilde V_ke^{6\eta_k}-120 e^{6\eta}-720\tilde\ve_k\psi e^{6\eta}\right)dx=o(\tilde\ve_k),
\end{equation}
uniformly for $r\in (0,\tfrac{1}{10r_k})$, where $\tilde V_k(x):=V_k(r_kx)$.
\end{lem}
\begin{proof}
We prove the lemma in few steps. 

\noindent\textbf{Step 1} $u_k(r)\leq -\frac78 u_k(0)$ on $ (t_k,\frac{1}{10})$, where $t_k$ is as in Lemma \ref{neck1} for some $p\in (1,2)$. 

It follows from Lemma \ref{monotone},  $iii)$ of Lemma \ref{neck1}, \eqref{betakuk} and \eqref{ukbetak}  that  $$u_k(r)\leq \max \{u_k(t_k), u_k(\tfrac{1}{10})\}\leq -\tfrac78 u_k(0)\quad\text{for every }r\in (t_k,\tfrac{1}{10}),$$ for $k$ large. 

\noindent\textbf{Step 2} We set $s_k:=\tilde\ve_k^{-\frac15}$ and claim that \eqref{stimeinteg} holds for $r\in (0,s_k)$.

Indeed from Lemma \ref{2.8}
$$\eta_k=\eta+\tilde\ve_k\psi+\tilde\ve_k\delta_k\phi_k=\eta+\tilde\ve_kO(1+r^4)+\tilde\ve_k\delta_kO(1+r^4),\quad r\in(0,s_k),$$
which yields $\eta_k-\eta=o(1)$ on $(0,s_k)$. Therefore,  for every $r\in (0,s_k)$ 
\begin{align*}
\int_{B_r}\left(\tilde V_ke^{6\eta_k}-120e^{6\eta}\right)dx
&=  \int_{B_r}e^{6\eta}\left(\tilde V_ke^{6(\eta_k-\eta)}-120\right)dx\\
  &= \int_{B_r}e^{6\eta}\left((120+O(r_k^2|x|^2))(1+(6+o(1))(\eta_k-\eta)-120\right)dx\\
  &=\tilde\ve_k\int_{B_r}e^{6\eta}\left((720+o(1))\psi+o(\phi_k)+o(|x|^2)\right)dx\\
 &=720\tilde\ve_k\int_{B_r}e^{6\eta}\psi dx +\tilde\ve_k\int_{B_r}e^{6\eta}o(1+|x|^4)dx \\
  &=720\tilde\ve_k\int_{B_r}e^{6\eta}\psi dx +o(\tilde\ve_k).
 \end{align*}

\noindent\textbf{Step 3} We claim that
$$\int_{B_\frac{1}{10r_k}\setminus B_{s_k}}\tilde V_ke^{6\eta_k}dx=o(\tilde\ve_k).$$ 
We write  (if $\frac{t_k}{r_k}\leq s_k$ then the second integral $I_2$ is considered to be $0$) 
$$ \int_{B_\frac{1}{10r_k}\setminus B_{s_k}}\tilde V_ke^{6\eta_k}dx\leq \int_{B_\frac{1}{10r_k}\setminus B_\frac{t_k}{r_k}}\tilde V_ke^{6\eta_k}dx+ \int_{B_\frac{t_k}{r_k}\setminus B_{s_k}}\tilde V_ke^{6\eta_k}dx=:I_1+I_2.$$
By Step 1
$$I_1=\int_{B_\frac{1}{10}\setminus B_{t_k}}V_ke^{6u_k}dx\leq Ce^{-\frac{21}{4}u_k(0)}=o(r_k^5)=o(\tilde\ve_k).$$
Using that $r^pe^{u_k(r)}$ is monotone decreasing on $(c_pr_k,t_k)$, we bound 
\begin{align*}
 \int_{B_\frac{t_k}{r_k}\setminus B_{s_k}}\tilde V_ke^{6\eta_k}dx\leq C s_k^6e^{6\eta_k(s_k)}\leq C s_k^6e^{6\eta(s_k)}=O(s_k^{-6})=o(\tilde\ve_k),
\end{align*}
where in the last inequality we have used that $\eta_k(s_k)=\eta({s_k})+o(1)$ by Lemma \ref{2.8}.

\noindent\textbf{Step 4} To complete the proof it remains to show that
$$\int_{B_{s_k}^c}e^{6\eta}dx=o(\tilde\ve_k),\quad  \int_{B_{s_k}^c}\psi e^{6\eta}dx=o(1).$$
The first estimate follows from
\begin{align*}
 \int_{B_{s_k}^c}e^{6\eta}dx=\int_{B_{s_k}^c}\left(\frac{2}{1+|x|^2}\right)^6dx=O(s_k^{-6})=o(\tilde\ve_k).
\end{align*}
The second one follows from Proposition \ref{psi-vol} since $\psi(x)=O(|x|^4)$ as $|x|\to\infty$ implies
$$\int_{B_{s_k}^c}\psi e^{6\eta}dx=\int_{B_{s_k}^c}O\bra{|x|^{-8}}dx=O(s_k^{-2})=o(1).$$
\end{proof}

\begin{lem}\label{psi-nontrivial}
We have $\psi(r)=ar^2+O(\log r)$ as $r\to\infty$, where
\begin{equation}\label{formulaa}
a=\lim_{k\to\infty}\frac{2r_k^2\beta_k}{\tilde\ve_k}>0.
\end{equation}
\end{lem}
\begin{proof}
We proceed by steps. 

\medskip
\noindent\textbf{Step 1} $\psi(r)=ar^2+O(\log r)$ at infinity. 

We assume by contradiction that $\psi(r)=ar^2+br^4+O(\log r)$ at infinity for some $b\neq 0$. 
From Lemma \ref{vol-psik} we get
$$\int_{B_t}\left(\D^3\psi_k(x)-\D^3\psi(x)\right)dx=o(1),\quad t\in (0,\tfrac{1}{10r_k}),$$
hence, also using that $\psi_k\to\psi$ in $C^5_{\loc}(\R^6)$, we infer
\begin{align}
  |\D^2\psi_k(r)-\D^2\psi(r)|&\leq  |\D^2\psi_k(1)-\D^2\psi(1)|+\int_1^r\frac{1}{\omega_5t^5}\left|\int_{B_t}\left(\D^3\psi_k(x)-\D^3\psi(x)\right)dx\right|dt\notag\\
  &=o(1),\quad \text{uniformly for }r\in (1,\tfrac{1}{10r_k}).\label{clm2} 
   \end{align}
Since $\frac{\theta_{4,k}}{r_k}\to\infty$, from  \eqref{psiasymbis} we get $\Delta^2\psi(\theta_{4,k})=384b+o(1)$. 
 Taking $r=\frac{\theta_{4,k}}{r_k}$ in \eqref{clm2} and recalling that $\D^2 u_k(\theta_{4,k})=0$ 
\begin{align}
 o(1)&=\frac{1}{\tilde\ve_k}\left(\D^2\eta_k(\tfrac{\theta_{4,k}}{r_k})-\D^2\eta(\tfrac{\theta_{4,k}}{r_k})\right)-\D^2\psi(\tfrac{\theta_{4,k}}{r_k})\notag\\
 &=-(32+o(1))\frac{r_k^4}{\tilde\ve_k\theta_{4,k}^4}-(384b+o(1)). \notag
\end{align}
Recalling that $\theta_{4,k}^{-4}\approx \beta_k$ and $r_k^2\beta_k\to0$, this implies
$$\tilde\ve_k\approx \frac{r_k^4}{\theta_{4,k}^4}\approx r_k^4\beta_k, $$ a contradiction to $\tilde\ve_k\geq r_k^2\sqrt{\beta_k}$. 

\medskip
\noindent\textbf{Step 2} For any $L>0$
\begin{align} \label{13}\D \psi_k(r)-\D \psi(r)=(-12+o(1))\frac{\beta_k r_k^4}{\tilde\ve_k}(r^2-L^2) +O(L^{-2})+o(1),\quad r\in (L,\tfrac{1}{10r_k}).
\end{align}

Since $\psi(r)=ar^2+O(\log r)$,  { from Proposition \ref{psi-vol}} we have  $$|\D^2\psi(x)|\leq \frac{C}{1+|x|^4},\quad x\in\R^6.$$
From \eqref{ukbetak}
$$\D^2\psi_k(\tfrac{1}{10r_k})= -(384+o(1))\frac{\beta_k r_k^4}{\tilde\ve_k},$$
hence, for $r\in (0,\frac{1}{10r_k})$
\begin{align*}
\D^2\psi_k(r)&=\D^2\psi_k(\tfrac{1}{10r_k})-\int_{r}^{\frac{1}{10r_k}}\frac{1}{\omega_5t^5}\int_{B_t}\D^3\psi_kdxdt\notag\\
 &=(-384+o(1))\frac{\beta_k r_k^4}{\tilde\ve_k}+O(\tfrac{1}{r^4}),\notag\\
\end{align*}
where in the second equality we have used that 
 $$\left|\int_{B_r}\D^3\psi_k dx\right|\leq C,\quad r\in(0,\frac{1}{10r_k}),$$ which is a consequence of Lemma \ref{vol-psik}. 
 Therefore,  for any $L>0$ and $r\in (L,\frac{1}{10r_k})$ 
\begin{align}
 &\D \psi_k(r)-\D \psi(r)\notag\\
 &=\D \psi_k(L)-\D\psi(L)+\int_{L}^{r}\frac{1}{\omega_5t^5}\int_{B_t}(\D^2\psi_k-\D^2\psi)dxdt\notag\\
 &=o(1)+\int_{L}^{r}\frac{1}{\omega_5t^5}\int_{B_t}\left((-384+o(1))\frac{\beta_k r_k^4}{\tilde\ve_k}+O(\frac{1}{|x|^4}) +O(\frac{1}{1+|x|^4}) \right)dxdt\notag\\
 &=(-12+o(1))\frac{\beta_k r_k^4}{\tilde\ve_k}(r^2-L^2) +O(L^{-2})+o(1).  \notag
\end{align} 

\medskip 

\noindent\textbf{Step 3} $a\neq 0$.  

We assume by contradiction that $a=0$. Then $\psi $ is of the form $\psi(r)=c_0\frac{1-r^2}{1+r^2}$ for some $c_0\in\R$, thanks to Proposition \ref{psi-vol}. Since $\psi(0)=0$, we must have $c_0=0$, that is, $\psi\equiv 0$. Therefore, by Remark \ref{rem1}   we have $\tilde\ve_k=r_k^2\sqrt{\beta_k}$.

Taking $r=\frac{\theta_{2,k}}{r_k}$ in \eqref{13} and using that $\D u_k(\theta_{2,k})=0$, $\theta_{2,k}^2\beta_k\to \frac13$, we obtain 
\begin{align*}
&\frac{1}{\tilde \ve_k}(8+o(1))\frac{r_k^2}{\theta_{2,k}^2}=(-12+o(1))\frac{\beta_k r_k^4}{\tilde \ve_k}\frac{\theta_{2,k}^2}{r_k^2}+O(R^{-2})+o(1) \\
 i.e., \quad  &(24+o(1))\sqrt{\beta_k}=O(1),
\end{align*}
a contradiction. 

\medskip 

\noindent\textbf{Step 4}  $a>0$.

Since $a\neq 0$, we can choose $L>0$ large such that $|O(L^{-2})|\leq \frac12|\D\psi(\infty)|=6|a|$.
Taking $r=\frac{\theta_{2,k}}{r_k}$ in \eqref{13} and using that $\theta_{2,k}^2\beta_k\to \frac13$, one gets
\begin{align*}
\D \psi(\infty)+O(L^{-2}) +o(1)&= \frac{1}{\ve_k}\left((12+o(1))\beta_kr_k^2\theta_{2_k}^2+(8+o(1))\frac{r_k^2}{\theta_{2,k}^2}\right)\\
&= \frac{1}{\tilde \ve_k}(24+o(1))\beta_kr_k^2.
\end{align*}
This shows that 
 $$0\neq 12a=\D\psi(\infty)=24\lim_{k\to\infty}\frac{\beta_kr_k^2}{\tilde \ve_k}\geq0.$$
We conclude the lemma.
\end{proof}

\paragraph{Proof of \eqref{etakexp}-\eqref{stimapsi} (completed).}
Thanks to \eqref{formulaa} we have
\begin{equation}\label{vektilde}
\ve_k=(\tfrac{a}{8}+o(1))\tilde\ve_k,
\end{equation}
hence, also using Proposition \ref{psi-vol},
$$\ve_k^{-1}(\eta_k-\eta) \to \psi_0(x)=8|x|-48\log|x|+o(\log|x|), \quad \text{as }|x|\to\infty.$$
Then we have
$$\eta_k=\eta_0+\tilde\ve_k \psi+o(\tilde \ve_k)\phi_k=\eta_k+\ve_k\psi_0+o(\ve_k)\psi_0+o(\ve_k)\phi_k.$$
and by Lemmas \ref{2.8} and \ref{psi-nontrivial} we have $\psi_0(x)+\phi_k(x)=O(1+|x|^4)$ on $B_{s_k}$ for a given sequence $(s_k)$ with $s_k=o(\ve_k^{-\frac{1}{4}})$.
\hfill$\square$

\section{Proof of Theorem \ref{thm-curv}}



By Lemmas \ref{vol-psik} and \ref{psi-nontrivial} and also using \eqref{vektilde} and
\begin{equation}\label{defint}
\int_0^r \frac{s^5}{(1+s^2)^6}dr=\frac{1}{60}\bra{1-\frac{10r^4+5r^5+1}{(1+r^2)^5}},
\end{equation}
for $\ve_k=\frac{u_k(0)} {e^{2u_k(0)}}$ we now have
\begin{align*}
\int_{B_r}\tilde V_k e^{6\eta_k}dx&=120\int_{B_r}e^{6\eta}dx+\ve_k 720 \int_{B_r}\psi_0 e^{6\eta}dx+o(\ve_k)\\
&=I_{1,r}+\ve_k I_{2,r}+o(\ve_k), \quad \text{for }r\in (0,\tfrac{1}{10r_k}).
\end{align*}
Using \eqref{defint} we obtain
$$I_{1,r}=\Lambda_1\bra{1-\frac{10}{r^6}+o(r^{-6})}.$$
From Proposition \ref{psi-vol} we get
$$I_{2,r}=720\int_{\R^6}\psi_0 e^{6\eta}dx+o(1)=24\Lambda_1+o(1),\quad o(1)\xrightarrow{r\to\infty} 0.$$
Since for $r\ge r_k^{-\frac13}$ we have $r^{-6}\le r_k^2=4e^{-2u_k(0)}=o(\ve_k)$, we obtain
$$\int_{B_r}\tilde V_k e^{6\eta_k}dx=\Lambda_1 +24\Lambda_1 \ve_k +o(\ve_k)\quad r\in \bra{\tfrac{1}{\sqrt[3]{r_k}}, \tfrac{1}{10r_k}},$$
and scaling back we obtain
$$\int_{B_r}V_k e^{6u_k}dx=\Lambda_1 +24\Lambda_1 \ve_k +o(\ve_k),\quad \text{for } r\in \bra{r_k^{\frac23}, \tfrac{1}{10}}.$$
Finally, using that for $\delta<\delta^*=\sqrt{1-\frac{1}{\sqrt3}}$ we have $(1-\delta^2)^2>\frac{1}{3}$, and 
\begin{align*}
u_k(x)&=-u_k(0)(1-|x|^2)^2(1+o(1))\\
&\le -u_k(0)(1-\delta^2)^2+o(u_k(0)), \quad \text{for }|x|\in \bra{\tfrac{1}{10},\delta}
\end{align*}
we infer
$$\int_{B_{\delta}\setminus B_{\frac{1}{10}}}V_k e^{6u_k}dx =o(\ve_k),$$
and \eqref{curvsharp}  follows at once.
\hfill $\square$



\section{Proof of Proposition \ref{psi-vol}}\label{proofpsi}

We prove the proposition in few steps. 

\medskip \noindent\textbf{Step 1} We claim that $\psi(x)=O(|x|^4)$ as $|x|\to\infty$.

Choose $r_0> 1$ such that
$$\int_{|x|>r_0}e^{6\eta}(1+|x|^4)dx<\ve,$$
where $\ve>0$ will be fixed latter.
We set
$$X:=\left\{\phi\in C^0([r_0,\infty)): \|\phi\|<\infty\right\},\quad \|\phi\|:=\sup_{[r_0,\infty)}\frac{|\phi(x)|}{(1+|x|^4)}.$$
Let $T:X\to X$, $T(\phi):=\bar{\phi}$ where $\bar\phi$ is the unique solution to the ODE
$$\D^3\bar{\phi}=-720e^{6\eta}\phi, \quad \bar{\phi}^j(r_0)=\psi^j(r_0),\,j=0,1,...,5.$$ 
Notice that for $f\in C^2_{rad}$ one has
\begin{align}\label{int-repre}
 f(r_1)=f(r_0)+r_0^5f'(r_0)\int_{r_0}^{r_1}\frac{dr}{r^5}+\int_{r_0}^{r_1}\frac{1}{\omega_5r^5}\int_{r_0<|x|<r}\D fdxdr,\quad 0<r_0<r_1. 
\end{align}
A repeated use of \eqref{int-repre} with $f=\D^2\bar\phi$, $\D\bar\phi$ and $\bar\phi$ gives 
 $$|\bar{\phi}(t)|\leq C_1(1+t^4)+C_2\ve\|\phi\|t^4,\quad t\geq r_0,$$ where $C_1=C_1(r_0)$ depends on the initial conditions $\psi^j(r_0)$ and $C_2$ is a dimensional constant. 
Therefore, for $C_2\ve<\frac12$ and $M>2C_1$ we have
\begin{align}
\|\bar{\phi}\|=\sup_{[r_0,\infty)}\frac{|\bar{\phi}(x)|}{(1+|x|^4)}\leq C_5(r_0)+C_6\ve\|\phi\|\leq M,\quad \text{for } \phi\in\mathcal{B}_M,\notag
\end{align}
where $\mathcal{B}_M:=\{\phi\in X:\|\phi\|\le M\}$. Thus, $T:\mathcal{B}_M\to \mathcal{B}_M$ and by the Schuder fixed point theorem, $T$ has a fixed point $\psi_*\in \mathcal{B}_M$. 
From the uniqueness of  solutions we have $\psi_*=\psi|_{[r_0,\infty)}\in \mathcal{B}_M$, and this proves the claim. 

\medskip\noindent\textbf{Step 2} 
We claim that $\psi(x)=P(x)-\alpha\log|x|+o(\log|x|)$ for some $\alpha\in\R$, where $P$ is a radial polynomial of degree at most  $4$.

We set
$$\bar\psi(x):=\frac{720}{\gamma_6}\int_{\R^6}\log\left(\frac{1}{|x-y|}\right)e^{6\eta(y)}\psi(y)dy,$$
which is well-defined thanks to Step 1, and $P:=\psi-\bar\psi$. Then $(-\D)^3P=0$ on $\R^6$ and since $P$ is radially symmetric, $P$ is a polynomial of degree at most $4$, which we write as $P(r)=ar^2+br^4+d$.

The property $\bar\psi=\alpha\log|x|+o(\log|x|)$ with $|\nabla^jo(\log|x|)|=o(|x|^{-j})$ for $1\le j\le 5$ follows easily from its integral definition.


\medskip\noindent\textbf{Step 3} $\alpha= 6a+48b$. 

Since the function $\eta_\lambda(x):=\eta(\lambda x)+\log\lambda$ solves $(-\Delta)^3\eta_\lambda=120e^{6\eta_\lambda}$ in $\R^6$, one easily sees that the function
\begin{align}\label{Psi}
\Psi(x):=\frac{1-|x|^2}{1+|x|^2}=\frac{\partial \eta_\lambda(x)}{\partial\lambda}\bigg|_{\lambda=1}
\end{align}
satisfies 
$$(-\D)^3\Psi=720\Psi e^{6\eta}\quad\text{in }\R^6.$$
This shows that  $\psi\D^3\Psi=\Psi\D^3\psi$ on $\R^6$. Then with a repeated integration by parts one obtains for every $r>0$
\begin{align}\label{24}
0&=\int_{B_r}\psi\D^3\Psi dx-\int_{B_r}\Psi\D^3\psi dx\notag\\
&=\int_{\partial B_r}\left(\psi (\D^2\Psi)'-\psi'\D^2\Psi +\D\psi (\D\Psi)'-(\D\psi)'\D\Psi+\D^2\psi (\Psi)'-(\D^2\psi)'\Psi \right)d\sigma.
\end{align}
From the previous step and \eqref{Psi} it follows that 
\begin{align*}
&\psi(r)=ar^2+br^4+(-\alpha+o(1))\log r,\quad \Psi(r)=-1+\frac{2+o(1)}{r^2},\\
&\psi'(r)= 2ar+4br^3+\frac{-\alpha+o(1)}{r}, \quad \Psi'(r)= -\frac{4}{r^3}+\frac{8+o(1)}{r^5},\\
&\D\psi(r)= 12a+32br^2+\frac{-4\alpha+o(1)}{r^2}, \quad \D\Psi(r)=-\frac{8}{r^4}+\frac{24+o(1)}{r^8},\\
& (\D\psi)'(r)= 64br+\frac{8\alpha+o(1)}{r^3}, \quad (\D\Psi)'(r)= \frac{32}{r^5}-\frac{192+o(1)}{r^9},\\
&\D^2\psi(r)=384b+\frac{16\alpha+o(1)}{r^4}, \quad \D^2\Psi(r)=\frac{768+o(1)}{r^{10}},\\
& (\D^2\psi)'(r)=\frac{-64\alpha+o(1)}{r^5}, \quad  (\D^2\Psi)'(r)=-\frac{7680+o(1)}{r^{11}},
\end{align*}
where $o(1)\to0$ as $r\to\infty$.
Plugging these estimates in \eqref{24} one obtains $\alpha=6a+48b$. 
 
 \medskip
 
 \noindent\textbf{Step 4} We prove that $\psi(r)=\gamma\frac{1-r^2}{1+r^2}$ when $a=b=0$. 
 
In this case, from Step 2 we can write $\psi=\bar\psi+d$. 
Indeed, by Step 3 $\alpha=0$, so that 
$$\bar\psi(x)
=\frac{720}{\gamma_6}\int_{\R^6}\log\left(\frac{1}{|x-y|}\right)e^{6\eta(y)}\psi(y)dy+\log|x|\frac{720}{\gamma_6}\underbrace{\int_{\R^6}e^{6\eta(y)}\psi(y)dy}_{=0},\quad x\ne 0$$
and we can write
$$\psi(x)=I_1(x)+I_2(x)+C,\quad x\neq 0,$$
where for $i=1,2$
$$I_i(x):=\frac{720}{\gamma_6}\int_{A_i}\log\left(\frac{|x|}{|x-y|}\right)e^{6\eta(y)}\psi(y)dy,\quad A_1:=B_1(x),\, A_2:=A_1^c.$$
For $|x|\geq 2$ we bound 
\begin{align*}
|I_1(x)|\leq C\|e^{6\eta}\psi\|_{L^\infty(A_1)}\int_{A_1}(\log|x|+|\log|x-y||)dy\leq C,
\end{align*}
and using that $$\frac{1}{1+|y|}\leq \frac{|x|}{|x-y|}\leq1+|y|\quad\text{for every }y\in A_2, \;|x|\ge 2,$$
one gets $$|I_2(x)|\leq C\int_{A_2}\log(1+|y|)e^{6\eta(y)}|\psi(y)|dy\leq C.$$
Thus, $\psi$ is bounded in $\R^6$.
    
Bounded solutions of \eqref{eqpsi} have been classified in \cite[Theorem 2.6]{Wei-Xu}, hence the proof of Step 4 is complete.

\section{Some examples}\label{examples}

 It is easy to verify that the cases from $i)$ to $iii)$ of Theorem \ref{thm-2} can actually occur. We will show a few examples.

\medskip\noindent\textbf{Example 1} Let $u$ be a solution to
\begin{equation}\label{liouentire}
(-\Delta)^3 u= 120e^{6u}\quad\text{on }\R^6,\quad \Lambda:=\int_{\R^6}120 e^{6u}dx<\infty.
\end{equation}
Such solutions exist for every $\Lambda>0$, thanks to \cite{CC,HD,H-volume,LM-vol}. Take
$u_k(x):=u(\tfrac{x}{k})-\log k$, which is also a solution to \eqref{liouentire}. Then this sequence $(u_k)$ is as in case $i)$ of Theorem \ref{thm-2}, with $u_k\to-\infty$ uniformly in $\R^6$.

If we set $u_k(x):=u(kx)+\log k$, then we are in case $ii)$ of Theorem \ref{thm-2}, with $u_k\to -\infty$ uniformly locally away from $0$ and $u_k(0)\to\infty$.

\medskip

As the next example shows, things can get more complicated.
 



\medskip\noindent\textbf{Example 2} Another example of case $ii)$ of Theorem \ref{thm-2} is as follows. Let $(u_k)$ be as in Theorem \ref{thm-1} for some given $\Lambda\ge \Lambda_1$. Then we can choose $\rho_k\to \infty$ slowly enough, such that for $v_k(x):=u_k(\rho_kx)+\log\rho_k$ there exists radii $0<s_k<t_k \to 0$ with $v_k(t_k)\to\infty$, $v_k(s_k)\to-\infty$ and $v_k(0)\to\infty$.

 
 \medskip\noindent\textbf{Example 3} One can also construct an example of case $ii)$ of Theorem \ref{thm-2} in which $u_k(0)\to-\infty$ and $u_k(\rho_k)\to\infty$ for some $0< \rho_k \to 0$.  
 Fix $\Lambda>0$. Take $u_k(x)=v_k(\beta_kx)+\log \beta_k$ where  $v_k$ is a radial solution to 
 \begin{align*}
 v_k(x)&=\frac{120}{\gamma_6}\int_{\R^6}\log\left(\frac{1}{|x-y|}\right)e^{6v_k(y)}dy-\left(k+\frac{|\D v_k(0)|}{24}\right)(1-|x|^2)^2+c_k\\
 &=:I_k(x)-\left(k+\frac{|\D v_k(0)|}{24}\right)(1-|x|^2)^2+c_k,
 \end{align*}
with $$120\int_{\R^6}e^{6v_k}=\Lambda,$$ and $\beta_k=k+\frac{|\D v_k(0)|}{24}$.  
Existence of such $v_k$ can be proven in the spirit of \cite{H-volume}. In fact, one can show that $I_k=O(1)$ in $B_\delta$ for some $\delta>0$, $v_k(1)\to\infty$, $c_k\to\infty$, $c_k\ll\beta_k$ (see also \cite{HIM}). 
This example can be slightly modified to have $u_k(0)=0$ and $u_k(\rho_k)\to\infty$.
 
 



\section{Open questions}\label{OQ}

It is natural to ask what happens in the non-radial case, already in dimension $4$. In the very related case of the mean-field equation
\begin{equation}\label{MF}
(-\Delta)^m u_k =\rho_k\frac{e^{2m u_k}}{\int_{\Omega}e^{2mu_k}},\quad \text{in }\Omega
\end{equation}
with Dirichlet boundary conditions and the bound $\rho_k\le C$, using the Lyapunov-Schmidt reduction, several results have been produced, both in dimension $2$ (see e.g. \cite{BP, dPKM, EGP}), $4$ (see \cite{BDOP, CMM}) or higher (see \cite{Mor}). In this case one can construct solutions blowing up at finitely many points, which are located at a critical point of a so-called \emph{reduced} functional (compare to \cite{Rey}). The absence of polyharmonic blow-up for \eqref{MF} (contrary to case of Theorems \ref{ARSM}, \ref{thm-2} and \ref{thm-1}  is due to the Dirichlet boundary condition. In fact these existence results are the most general possible, see e.g. \cite{MP,RW}. On the other hand, in view of Theorems \ref{ARSM} and \ref{thm-1} we expect for \eqref{eq-1}-\eqref{eq-2} a large number of examples where both concentration and polyharmonic blow-ups occur.

\paragraph{General open question} For $n\ge 2$ take $\Omega\subset\R^{2n}$ open, a finite set $S_1\subset\Omega$, and $\varphi\in \mathcal{K}(\Omega,S_1)$. When is it possible to construct solutions to \eqref{eq-1}-\eqref{eq-2} having as blow up set exactly $S_1\cup S_\varphi$?

\medskip

More precisely, we can consider the following subquestion.

\medskip
\noindent\textbf{Open question 1} Is it necessary that the points in $S_1$ satisfy some balancing conditions, coincide with critical points of $\varphi$, or can they be prescribed arbitrarily?

\medskip
\noindent\textbf{Open question 2} If $S_\varphi\ne \emptyset$, should every blow up in $S_1$ be spherical?

\medskip
\noindent\textbf{Open question 3} Consider the following simple situation. Take
$$\varphi(x)=x_1-\frac{x_1^3}{3}-\frac{2}{3},\quad x=(x_1,x_2,x_3,x_4)\in \Omega:=(-2,2)\times\R^3\subset\R^4.$$
Then $\varphi\in \mathcal{K}(\Omega,\emptyset)$, $S_\varphi=\{1\}\times \R^3$ and $\nabla\varphi=0$ on $\{\pm1\}\times\R^3$. Is it possible for every finite set $A\subset\R^3$ to find solutions to \eqref{eq-1}-\eqref{eq-2} with $n=2$ and with polyharmonic blow-up on $S_\varphi$ and concentration blow-up at $S_1=\{-1\}\times A$?




\end{document}